\documentclass{amsart}

\usepackage{graphicx}              
\usepackage{amsmath}               
\usepackage{amsfonts}              
\usepackage{amsthm}                
\usepackage{amssymb}
\usepackage{amscd}
 \usepackage{listings}
\usepackage{times}
\usepackage{comment}
\usepackage{mathtools}
\usepackage{tikz}

\usepackage{stmaryrd}
\usepackage{xcolor}

\theoremstyle{definition}
\newtheorem{thm}{Theorem}[section]
\newtheorem{lem}[thm]{Lemma}
\newtheorem{prop}[thm]{Proposition}

\newtheorem{defn}[thm]{Definition}

\newtheorem{definition}[thm]{Definition}

\newtheorem{rem}[thm]{Remark}

\newcommand\QQ{{\mathbf Q}}
\newcommand{\ZZ}{\mathbf{Z}}      
\newcommand\CC{{\mathbf C}}
\newcommand{\Gm}{\mathbf{G}_m}
\newcommand{\kk}{\mathsf{k}}
\newcommand{\norm}{\mathrm{norm}}

\newcommand{\adele}{\mathbb{A}}

\newcommand{\fltns}{{\mathbin{\mkern-6mu\fatslash}}}

\newcommand{\Z}{\mathbf{Z}}
\newcommand{\Q}{\mathbf{Q}}
\newcommand{\R}{\mathbf{R}}
\newcommand{\SL}{\mathrm{SL}}
\newcommand{\C}{\mathbf{C}}
\newcommand{\F}{\mathbf{F}}

\newcommand{\GL}{\mathrm{GL}}
\newcommand{\PGL}{\mathrm{PGL}}

\newcommand{\wdd}{numerically weakly dual }

\newcommand{\gal}{\operatorname{Gal}}
\newcommand{\spec}{\mathrm{Spec}\, }

\newcommand{\Ggr}{\mathbf{G}_{\mathrm{gr}}}
\newcommand{\GGm}{\mathbf{G}_m}
\newcommand{\Ga}{\mathbf{G}_a}

\usepackage{hyperref}
\hypersetup{
    colorlinks,
    citecolor=blue,
    filecolor=blue,
    linkcolor=blue,
    urlcolor=blue
}

\begin{document}

\author{Eric Y. Chen and Akshay Venkatesh}
\title{Some singular examples of relative Langlands duality}
\begin{abstract}
Relative Langlands duality structures the study of automorphic periods
around a putative duality between certain $G$-spaces and $\check{G}$-spaces,
where $G, \check{G}$ are dual reductive groups. 
 
In this article, after giving a self-contained exposition of the relevant ingredients from relative Langlands duality, we examine this proposal for some interesting pairs of singular spaces: one pair arising from the cone of nilpotent $3 \times 3$ matrices, and the other pair arising from the nilpotent cone of $(2,2,2)$-tensors. These relate, respectively,  to Rankin--Selberg integrals 
discovered by Ginzburg and Garrett.

\end{abstract}

\maketitle
\tableofcontents

 \section{Introduction}

\subsection{Relative Langlands duality}
Let $G$ be a reductive group over $\CC$, and $\check{G}$ its Langlands dual group.   ``Relative Langlands duality'' (\cite{BZSV}, \S \ref{Hamact}) contends that Langlands duality should extend to a duality between certain smooth $G$ and $\check{G}$-spaces.
 The resulting class of pairs $(G,X)$ and $(\check{G}, \check{X})$   
supply many interacting structures in the
 the local and global Langlands program, among which is 
 a pair of equalities of the shape: 
\begin{equation} \label{introeq} \langle \mbox{$X$-Poincar{\'e} series}, f_G  \rangle  \sim L(\check{X}, f_G),
\mbox{$f_G$ a cusp form on $G$;} \end{equation} 
\begin{equation} \label{introeqd} 
 \langle \mbox{$\check{X}$-Poincar{\'e} series}, f_{\check{G}}  \rangle  \sim L(X, f_{\check{G}}),
 \mbox{$f_{\check{G}}$ a cusp form on $\check{G}$.}
 \end{equation}
 
 We will explain the notation here momentarily,
and in Theorem \ref{thm: smooth} we give, by way of illustration, how several standard examples
 of period integrals fit into this framework. 
 Our main result in this paper, Theorem \ref{mainthm}, exhibits
 two examples of pairs {\em with $X$ and $\check{X}$ singular} for which
 \eqref{introeq} and \eqref{introeqd} hold.  This falls outside the framework of \cite{BZSV}, in which the spaces were always smooth;  as we will see, the singular
 case introduces  several interesting new phenomena.  
 
 Let us sketch now what \eqref{introeq} and \eqref{introeqd} mean. 
 In both cases,  the left hand sides are defined in a standard way, to be recalled in \S \ref{autperiods}.
 However, their  right hand sides  are less standard.
 If $\check{X}$ were a linear representation of $\check{G}$, then $L(\check{X}, f_G)$ is 
simply the a special value of the $L$-function 
attached by Langlands (\cite{LanglandsWeil}) to the form $f_G$ and the representation of $\check{G}$ on $\check{X}$.
 However,
in most interesting examples,    $X$ and $\check{X}$ are {\em not} vector spaces, and the right hand side $L(\check{X}, f_G)$ is what one might call a ``nonlinear $L$-function.''
They will be outlined below in \S \ref{RS0}, 
and will be explained more carefully in \S \ref{spectralperiods}.

Now we indicate the geometry of the two examples
to be studied in Theorem \ref{mainthm}, see \S \ref{Examples} for more. 
 
  \begin{itemize}
 \item In one example, $G$ is isogenous to $\PGL_3 \times \Gm$, and $X$  
 is identified with the nilpotent
  cone for $\PGL_3$. Moreover,  $\check{X}$  is (generically) a triple cover of the same nilpotent cone which resolves its singularities in codimension $2$.   In this example, the second equality \eqref{introeqd} is
  essentially Ginzburg's integral representation \cite{Ginzburg} for the adjoint $L$-function for $\GL_3$;
  the dual equality \eqref{introeq} is new, and it does not represent a ratio of $L$-functions.

  \item In the other example, $G$ is isogenous to
  $\SL_2^3 \times \mathbf{G}_m$ and  $X$ will be the
space of $2 \times 2 \times 2$ tensors of hyperdeterminant zero. Moreover, $\check{X}$
is the affine cone over the Lagrangian Grassmannian of $3$-planes in $6$-space. 
In this example, \eqref{introeqd} is essentially Garrett's integral representation \cite{Garrett} of the triple
product $L$-function for $\GL_2$; the dual equality \eqref{introeq} is again new, not representing a ratio of $L$-functions. 
 \end{itemize} 
 
In both cases, much of the structure we are using is 
already  implicit or explicit in the original papers \cite{Ginzburg, Garrett} of Ginzburg and Garrett
on the automorphic side.  What our presentation brings out, however, is the existence
of a {\em symmetry} between automorphic and spectral sides, incarnated by the
``mysterious'' extra period identity \eqref{introeq}. This ``extra'' identity 
{\em does not represent} a quotient of standard $L$-values; indeed, the function occurring
therein does not even have meromorphic continuation to the full complex plane. 
  
It is worth observing that both examples can be described in a  uniform way: 
 $$ X = \mbox{affine closure of }\Gm U^0 \backslash G,
 \check{X} = \mbox{affine closure of }\Gm \check{U}^0 \backslash \check{G},$$
 where $U^0$ is the kernel of a natural map $U \rightarrow \mathbf{G}_a$
from a maximal unipotent subgroup $U$ to $\Ga$, with $\Gm \subset G$   
 normalizing $U^0$, and similarly on the dual side. We exploit this common structure to give uniform proofs. However,
 the fact that these are dual, in the sense of \eqref{introeq}, \eqref{introeqd}
 is not a general fact -- it uses numerical coincidences specific to the groups $G$ mentioned above. 
 
 \subsection{Hamiltonian actions} \label{Hamact}

 Next, let us briefly discuss how this relates to the notation and terminology of \cite{BZSV}. 
There the authors introduce a class of symplectic varieties with reductive group action (``hyperspherical Hamiltonian actions"), and a notion of duality
interchanging  symplectic $G$-varieties with symplectic $\check{G}$-varieties of this class -- that  is to say, a proposed duality
\begin{equation} \label{Mduality}  (G, M) \leftrightarrow  (\check{G}, \check{M})\end{equation}
built atop the Langlands duality $G \leftrightarrow \check{G}$. 

The general formulations of \cite{BZSV} simplify when both $M$ and $\check{M}$
can be polarized, i.e. written as $M=T^*X$ and $\check{M} = T^* \check{X}$
for $G$-spaces $X$ and $\check{G}$-spaces $\check{X}$, in which case the various predictions of \textit{loc. cit} can be rewritten in terms of $X, \check{X}$ and typically take a simpler form. 
Here, we will  always work at the level of $X, \check{X}$, avoiding  the question of giving a formal meaning to the cotangent bundles of singular spaces.\footnote{As a small warning,  we note that the choice of polarization is not unique, and therefore
\eqref{Mduality} will not correspond to an exact duality at the level of $X, \check{X}$. }

There are many examples in the automorphic literature 
which suggest that the general phenomena of \cite{BZSV} extend beyond the hyperspherical case,
although perhaps not in such a clean form. To understand the correct formulations,
a detailed study of examples is necessary, and this paper can be considered as a first step in this direction.

There are, indeed, significant and interesting differences 
between the singular and smooth case -- both at the level of statement and at the level of proof. For an example of the former, consider the
discrepancy in  powers of the discriminant $\Delta$ occurring in the global period formulas  in the singular case (the factors $a_{12}$ and $a_{21}$ in Definition \ref{weakly dual def}).  
For an example of the latter, the reader can compare 
the proof of Theorem \ref{mainthm} with the corresponding proofs in Appendix \ref{smoothproofs} in the smooth case; the  structure of integral points on $X$,
which  in turn controls the behavior of the Poincar{\'e} series, is substantially more complicated
in the singular case, see Proposition \ref{intpointsX}.

%

%
%
%
%

\subsection{Nonlinear $L$-functions: motivation}
 \label{RS0}  
 To motivate the nonlinear $L$-functions that appear in \eqref{introeq} and \eqref{introeqd}, and in our main theorem, we consider the classical example of Rankin--Selberg integrals.
  
 Starting with 
  two level $1$ modular forms $\sum a_n q^n$ and $\sum b_n q^n$,
  Rankin and Selberg provided an integral representation of the Dirichlet series
$ L^*(s, f \otimes g) := \sum a_n b_n n^{-s} $. 
This $L^*(s)$ does not satisfy an elegant functional equation;
for that purpose, it is usually replaced by the
product  $L(s, f\otimes g) := L^*(s, f\otimes g) \cdot \zeta(2s)$. 
But for us the ``naive'' $\sum a_n b_n n^{-s}$
is in fact a prototype of a nonlinear $L$-function. 

Indeed, 
let $A_p, B_p \in \mathrm{SL}_2(\C)$ be the Satake matrices
for the two forms, so that $\mathrm{tr}(A_p)=a_p$ and $\mathrm{tr}(B_p) =b_p$, and note that:
$$ 
\mbox{local factor of $L(s, f \otimes g)$} = \mbox{trace of $(A_p \otimes B_p) \cdot p^{-s}$ acting on 
functions on $\mathbf{A}^2 \otimes \mathbf{A}^2$ }.$$
By contrast, writing $\mathcal{N}$ for the rank one tensors in $\mathbf{A}^2 \otimes \mathbf{A}^2$, 
 $$ 
\mbox{local factor of $  L^*(s, f \otimes g)$} = \mbox{trace of $(A_p \otimes B_p) \cdot p^{-s}$ acting on 
functions on $\mathcal{N}$ }.$$
 
This motivates
calling $L^*(s, f\otimes g)$ ``the $L$-function attached to $\mathcal{N}$ as a $\GL_2 \times \GL_2$-space''
just as $L(s, f \otimes g)$ is the $L$-function attached to $\mathbf{A}^2 \otimes \mathbf{A}^2$
as a $\GL_2 \times \GL_2$-representation.

More generally, if $f$ is an automorphic form on $G$, and 
$\check{X} $ is a {\em conical} $\check{G}$-variety,
we may then define
$L(s, f, \check{X})$ as the Euler product
whose $p$th Euler factor  is described thus:
$$ \mbox{local factor of $L(s, f, \check{X})$} =  \mbox{trace of $A_p  \cdot  p^{-s}$ acting on functions on $\check{X}$},$$
where $A_p \in \check{G}$ is a Satake parameter for $f$ at the prime $p$. 
The resulting Euler product will be convergent for $\mathrm{Re}(s)$ sufficiently large;
it will {\em not}, in general, extend to a meromorphic function.

This definition -- that is to say, for conical $\check{X}$ -- is sufficient to cover our main
result Theorem \ref{mainthm}; for the case of non-conical $\check{X}$,
we refer again to \S \ref{spectralperiods}. 
   The idea of
attaching $L$-functions to nonlinear spaces is a key suggestion of \cite{BZSV};
here, we bring out the point that the class of nilcones perhaps has particular interest.
We believe this viewpoint is a step towards a conceptual explanation of numerology first observed by Rallis and Ginzburg \cite{Ginzburg-Rallis}.

\subsection{Connection to work on singular spherical varieties}
 
There is a significant body of work that seeks to
study periods and harmonic analysis on {\em singular} spherical varieties. 
Perhaps the most closely related to this paper are proposals in the work of Sakellaridis \cite[Conjecture 3.2.2]{Sakellaridis-RSmethod}
(for more in this direction, see also the later work \cite{Bouthier-Ngo-Sakellaridis} of Bouthier, Ngo, and Sakellaridis,  as well as previous work  of Braverman--Kazhdan \cite{BK}). 

Speaking a little informally, these papers suggest that one can often equip the singular space $X$ with additional data -- an ``exotic'' basic function, constructed
using intersection cohomology -- and, {\em considered
in conjunction with this extra data}, the dual of $X$ is smooth. What we show here is that, at least in the two cases we study, 
one can proceed more naively, without using an exotic basic function. Although the dual of $X$ in this naive sense is singular,
  the resulting theory remains meaningful in the framework of relative Langlands duality
even if it is further from the traditional theory of $L$-functions.\footnote{Cf. Sakellaridis {\em op. cit.}, Example 3.2.1: ``
but the numerator does not represent an $L$-function and it would be unreasonable to expect that its Euler product admits meromorphic continuation.
Therefore, this was not the correct Schwartz space...''  What we show is, despite the absence of meromorphic continuation, there still 
seems to be an interesting theory.}
 See Remark \ref{Schwarzchoice} for more discussion on this topic.

%
%

%
%
%

\subsection{Notation} \label{notation} 
\subsubsection{Groups}
 $G$ and $\check{G}$ will denote a pair of Langlands dual reductive groups. 
 We will work with these over any ring, and by this we always
 mean the split (Chevalley) form, which we always regard as equipped with its standard pinning. 
 
 We denote by $T, \check{T}$ the split tori of $G, \check{G}$
 and by
 $$e^{2  \rho} \in X^*(T) = X_*(\check{T})$$
 the sum of positive roots; similarly we define $e^{2 \check{\rho}}$ to be the sum of positive coroots. 
 We denote by $U, \check{U}$ the  standard maximal unipotent subgroups of $G, \check{G}$.

\subsubsection{Coefficient fields} \label{Fandk}
$\F$ will be used for an ``automorphic'' coefficient field; it will always
be a finite field of size a prime power $q$. 
$\kk$ will be used for a ``Galois-side'' coefficient field; it 
can be taken to be the algebraic closure of the field of $\ell$-adic numbers.  It will be convenient to fix, once and for all, an isomorphism
$\kk \simeq \C$, so we will freely move between $\kk$-valued and $\CC$-valued notions without explicit comment.

 \subsubsection{Group actions and the grading group $\Ggr$} \label{Xdef}
 $G$ and $\check{G}$ will act on various varieties and schemes. 
Our convention for actions is as follows:
 \begin{quote}
Group actions on spaces on the right, and group actions on functions, forms etc.  are on the left,
derived from geometric actions by pullback.
\end{quote}
This convention differs from \cite{BZSV}, and we compare them in \S \ref{BZSVcomp}.

  Let $X, \check{X}$ be affine $\F$-varieties, 
 admitting actions of 
  $G \times \Ggr$
and $\check{G} \times \Ggr$ respectively.   Here, $\Ggr$ is simply another name for $\Gm$, but it plays a role quite distinct to any central torus in $G$ and to avoid confusion
it is best to label it differently -- the notation arises from its role as a ``grading group'' in the context of \cite{BZSV}.

For us, the most important
function of the $\Ggr$ action is that it will govern, in general, where the $L$-function is to be evaluated,
when $X$ or $\check{X}$ is on the spectral side. 
An  {\em eigenmeasure} is a
differential form of top degree on the smooth locus $X^{\circ}$
which is an eigenvector under the translation action of both $G$ and $\Ggr$: 
\begin{equation} \label{etadef}  (g, \lambda)^* \omega = \eta(g, \lambda) \omega = \eta(g)\lambda^\varepsilon \omega,\end{equation} 
for a character $\eta: G \times \Ggr \rightarrow \Gm$ and an integer $\varepsilon \in \ZZ$. 

Our primary interest is in the case when $X, \check{X}$ are {\em conical}, in the following sense.
\begin{defn}\label{defn: conical}
Let $k$ be a field, and let $X$ be an affine $k$-variety with $\Gm$-action. We say that $X$ is \textit{conical} if the coordinate ring $k[X]$ has only nonnegative $\Gm$-weights, and the 0th graded piece is isomorphic to $k$.
(In particular, such an $X$ has a unique $\Gm$-fixed point, usually to be denoted by $0$). 
\end{defn}
For our purposes, $k$ will either be $\F$ (the automorphic coefficient field) or $\kk$ (the Galois coefficient field). A $G \times \Ggr$-variety $X$ is \textit{conical} if it is conical in the above sense with $\Gm = \Ggr$.

 \subsubsection{Curves}
 
Let $\Sigma$ be a curve of genus $g$ over $\F$ and with function field $F$. 
We introduce the letter $\Delta$ for the discriminant of $\Sigma$, that is to say:
\begin{equation} \label{Deltadef}  \Delta :=  q^{2g-2}.\end{equation} 
We will write $\zeta(s)$ for the $\zeta$-function of $\Sigma$, i.e.
$\zeta(s) = \prod_{v} (1-q_v^{-s})^{-1}$ where the product ranges over places $v$
with residue field of size $q_v$.

We denote by $\mathbb{A}$ the adele ring of $F$
and by $\mathcal{O} \subset \mathbb{A}$ the maximal compact subring. 
Write \begin{equation} \label{bracketnotation} [G] := G_F \backslash G_{\mathbb{A}} / G_{\mathcal{O}}\end{equation} 
the adelic quotient, equivalently, the
set of isomorphism classes of  ($\F$-rational) $G$-bundles over $\Sigma$.
Similar notation is used for subgroups of $G$, e.g. $[U]$. 

 An ``unramified automorphic form'' on $G$
will be, by definition, a 
function $$f: [G] \rightarrow \kk$$
which is an eigenfunction of all Hecke operators. 

\subsubsection{Local notation}
We use $v$ for a place of $F$ and $F_v, \mathcal{O}_v, \varpi_v$ for the completion of $F$, the local ring of integers, and the uniformizer, at $v$. We write $q_v$ for the size of the local residue field at $v$. The normalized local valuation will be denoted by $x \mapsto |x|_v$ and
its product over all places gives the adelic valuation $\mathbb{A}^{\times} \rightarrow \R_{>0}$, often
denoted simply $x \mapsto |x|$. If $\chi$ is a cocharacter of some torus $T$, we write $\varpi^\chi$ for the element $\chi(\varpi) \in T(F_v)$.

\subsubsection{Galois parameters}

We will write
$$\Gamma \mbox{ or } \Gamma_F = \mbox{Weil group of $\Sigma$}$$
for the  unramified global Weil group of $\Sigma$, i.e. the preimage of integer powers of Frobenius inside the {\'e}tale fundamental
group (equivalently: everywhere unramified Galois group) of $\Sigma$.
This choice reflects the fact that we will work everywhere with unramified parameters
and unramified automorphic forms.

\subsubsection{Additive characters}

Let $\psi: \mathbb{A}/F \rightarrow \CC^{\times}$
be an additive character  
 whose conductor of $\psi$ at each place $v$ is even.
 That is to say, there exists, for each such $v$,
an {\em even} integer $2m_v$, with the property that 
$\psi$ is trivial on $\varpi_v^{-2m_v} \mathcal{O}_v$
but not on $\varpi_v^{-2m_v-1} \mathcal{O}_v$. 
Such a character exists by a theorem of Hecke (see \cite{Hecke} Satz 176).\footnote{ Let $K^{1/2} = K^{1/2}_\Sigma$ be a choice of $\F$-rational spin structure;
 it exists, by a theorem of Hecke, and we fix a rational section $\nu$ of $K^{1/2}$ which determines also a rational 1-form $\omega = \nu^{\otimes 2}$.
 Let $2m_v$ be the vanishing order of $\omega$ at a place $v$, 
 We take $\psi$ to be the character given by 
 determined by $\omega$, i.e. locally sending $f$ to $\mathrm{Res}(f \omega)$.} 

We write
 \begin{equation} \label{partialdef} \partial^{1/2} = (\varpi_v^{m_v}) \in \mathbb{A}^\times, \mbox{ so that } |\partial| =q^{-(2g-2)}.\end{equation}

%

\subsubsection{Measures} \label{groupmeasures}
For $G$ a reductive group over $\F$, we shall 
  normalize the Haar measure on $G(\mathbb{A})$ in such a way that 
it assigns mass $1$ to the standard maximal compact subgroup
$\prod_{v} G(\mathcal{O}_v)$. 

In the case of the additive group $\mathbf{G}_a$ there are two measures of interest:
firstly, the ``integral'' measure $\mu$ just described which assigns mass 1 
to the everywhere integral adeles; and the ``self-dual'' measure $\mu^{\psi}$ with respect to which $F \backslash \mathbb{A}$ has a self-dual Fourier transform with respect to $\psi$.
Locally, a self-dual lattice is given by $\varpi_v^{-m_v} \mathcal{O}_v$; its self-dual measure is $1$ whereas its ``integral'' measure
is $q_v^{m_v}$. 
  Thus, as an equality of measures on $\mathbb{A}$,
\begin{equation} \label{meas form}
\mu = \Delta^{1/2} \mu^{\psi}. 
\end{equation}
a familiar shadow of the fact that the ``volume'' of the ring
of integers in a number field is the square root of the discriminant. 

\subsubsection{Normalization of class field theory} 
Our normalization of the local class field theory
associates the modulus character $x \mapsto |x|$
of a local field $F$ with the cyclotomic character of its Galois group, that is to say, sending a geometric Frobenius element
to $q^{-1}$ where $q$ is the size of the residue field; or, said differently, the reciprocity map
of class field theory
sends a uniformizing element to geometric Frobenius. 
In what follows,  Frobenius means geometric Frobenius.

\subsection{Acknowledgements}
We would like to thank Paul Nelson for interesting conversations related to other examples, including the doubling integral.
A.V. would like to thank David Ben-Zvi and Yiannis Sakellaridis for many years of inspiring collaboration; some of their ideas
are also reflected in this paper. E.C. would like to thank David Ben-Zvi, Yiannis Sakellaridis, and the second named author for sharing their ideas and their encouragement regarding this project. 

The main ideas and results of this paper were presented (albeit in a much cruder form) in the 2023 Princeton PhD thesis of the first named author. Both authors acknowledge the support of the National Science Foundation. 
 
\section{Automorphic periods}  \label{autperiods}

Our goal here is to provide an exposition of the relevant parts of \cite{BZSV}; however,
 we will adapt and simplify the language to suit our current needs.

\subsection{The automorphic period attached to $X$}
We will specialize (and slightly adapt, by allowing singularities)
the discussion of \cite{BZSV} to the case at hand. 
We follow notation as in \S \ref{notation}; in particular, let $X$ be as in \S \ref{Xdef};
we shall assume it to admit an eigenmeasure. 

 We  consider the following unitarily-normalized action of 
 the adelic points of $G$ on the space of adelic Schwartz functions $\mathcal{S}(X(\mathbb{A}_F))$:
     $$g  \star \Phi(x) = |\eta(g)|^{1/2} \, \Phi(x g),$$
     where $\eta$ is as in \eqref{etadef}. 
    

To form the normalized theta series on $X$, we want to use not  the standard characteristic function of integral points on $X(\mathbb{A}_F)$
(which we call $\Phi^0$) but rather its translate through $\partial^{1/2}$:
\begin{equation} \label{Phidef} \Phi(x) :=  \mbox{right translate of $\Phi^0$ through $\partial^{1/2} \in \Ggr(\mathbb{A})$}
\end{equation} 

Let $\mathring{X} \subset X$ be an open subvariety;  
if $X$ is conical we shall take $\mathring{X} = X-\{0\}$,
and if otherwise unspecified  we always understand $\mathring{X} = X$. 

 The
normalized theta series of $X$ on $G(\mathbb{A})$ is defined by
    multiplying the Poincar{\'e} series $\sum_{x \in \mathring{X}(F)} \, (g,\partial^{1/2}) \star \Phi(x)$
    by a unitary normalization factor:
    \begin{align} \label{thetaXdef}
        \theta_X (g) &=  
        \Delta^{\frac{\dim X- \dim G}{4}} |\partial^{1/2}|^{1/2}  \sum_{x \in \mathring{X}(F)} \,  g \star \Phi(x),  
        \end{align} 
where $\Phi$ is as above, and $|\partial^{1/2}|$ is the factor by which
$\partial^{1/2} \in \Ggr(\mathbb{A})$ scales an eigenmeasure (in effect,
this latter factor amounts to replacing the action of $\Ggr$ in \eqref{Phidef} by a unitarily normalized action). 

\begin{rem} The fact that we sum over $\mathring{X}(F)$ instead of $X(F)$ in the definition is just a ``hack'' to avoid convergence problems coming from the volume of $[G]$, and more precisely the volume of any central $[\Gm]$.  A proper treatment of duality should include all of $X(F)$ in the definition of $\theta_X$
  and then regularize  its pairing with automorphic forms; in fact, the contribution of $0 \in X(F)$ should only pair nontrivially with the constant automorphic form.    
  Here we use it as a formal device to avoid some (trivial) divergences, and
  its role is so minimal that we do not introduce it in the notation explicitly. 
 \end{rem}

Using this theta series we can define the \textit{(normalized) automorphic $X$-period} for an unramified automorphic form $f$ on $G$ by 
   \begin{equation} \label{PXdef} P_X(f) := \int_{[G]} \, \theta_X(g)f(g) \, dg .
   \end{equation}  
   
\begin{rem}[Choice of Schwartz function] \label{Schwarzchoice} In defining $\theta_X$ and $P_X$, we have downplayed the role of the Schwartz function $\Phi$; in reality this is an extra degree of freedom that ought to be better understood. We have chosen $\Phi$ as in \eqref{Phidef} since we interpret $\theta_X$ as the function-theoretic avatar of ``counting $K^{1/2}$-twisted $X$-sections of a $G$-bundle", but it is not the only choice with geometric meaning. 

Notably, when $X$ is singular (which will be the most interesting case for us), another function of rapid decay termed the \textit{IC function} (see \cite{Bouthier-Ngo-Sakellaridis} and \cite{Sakellaridis-Wang}) is the function-theoretic analogue of the intersection cohomology of formal arc spaces (the latter of which is difficult to define). The IC-function gives a satisfying explanation to the extra zeta factors one needs in order to write down \textit{completed} Eisenstein series in the case when $X$ is the \textit{basic affine space} of a group, as corroborated by geometric calculations of \cite{Braverman-Gaitsgory};
and \cite{Sakellaridis-RSmethod} proposes that the IC function plays a similar role in the theory of integral representations of $L$-functions. 
On the other hand, in this article and in a companion article of the first author \cite{toric} we suggest that our choice of $\Phi$ is convenient for the study of relative Langlands duality, as it restores a pleasing symmetry between the $G$ and $\check{G}$-actions. 
\end{rem}

 The following lemma will handle  convergence  in all cases
 that we encounter here. 
 \begin{lem}
 Suppose $X$ is conical $G$-variety with respect to a central $\Gm \subset G$ (see Definition \ref{defn: conical})
 such that $G/\Gm$ is semisimple, and take $\mathring{X} = X-\{0\}$. 
 Then  $\int_{[G]} \, \theta_X(g)f(g) \, dg$ 
is absolutely and uniformly convergent if $f$ is a cusp form, and the real part of   the central character of $f$,
restricted to $\Gm$ is sufficiently large.  
\end{lem}

The restriction of the central character of $f$ to $\Gm$ defines a character $\mathbb{A}^{\times}/F^{\times} \rightarrow \C^{\times}$,
whose absolute value has the form $|x|^{s}$ for a unique $s \in \R$ --this is what
we have called the real part of the central character.

\proof
Let $Z \subset G$ be the central $\Gm$ with respect to which $X$ is conical
and let $0 \in X$ be the  unique fixed point of that $\Gm$ action. 

Being cuspidal, $f$ has compact support modulo center, i.e.
there exists a compact subset $\Omega \subset G_{\mathbb{A}}$
such that $f$ is supported inside $G_F \cdot \Omega \cdot Z(\mathbb{A})$. 
Since there are only a finite number of $G_{\mathcal{O}}$-orbits
on $\Omega \cdot G_{\mathcal{O}}$ the integral defining
$P_X(f)$ 
amounts to a finite sum of integrals over the adelic points of $Z$, each of which
is bounded 
by an integral of the form
$$ \int_{\lambda \in \mathbb{A}^{\times}/F^{\times}}  \left| \theta_X(\lambda g_0) \right|  |\lambda|^s d \lambda = \int_{\lambda \in \mathbb{A}^{\times}/F^{\times}}  |\lambda|^{s+c}
\sum_{x \in X(F) - \{0\}} \Phi(\lambda x)d\lambda$$
where the constant $c$ arises from the eigenmeasure character,
and we are permitted to restricted to $X_F-\{0\}$ because $0$ does not belong to $\mathring{X}$. 

Fix a place $w$. Since the image of $F_w^{\times}$ inside $\mathbb{A}^{\times}/F^{\times}$
is co-compact, it is sufficient to  analyze the same integral  $\mathbb{A}^{\times}/F^{\times}$ replaced by $F_w^{\times}$. 
Since the support of $\Phi$ is compact, it is enough to  
show that for each compact subset $\Omega \subset X(\mathbb{A})$
a bound of the form:
\begin{equation} \label{processed}  \left| \{v \in X_F-\{0\} : v \lambda \in \Omega  \}  \right| \leq  \begin{cases} |\lambda|^{-D},  & |\lambda| \leq c, \\ 0, &  |\lambda| \geq c \end{cases} \end{equation} 
for all $\lambda \in F_w^{\times}$ and for some constants $c, D$; because
the integral above is then bounded by $\int_{F_w^{\times}}  |\lambda|^{s-D}$ over the region $|\lambda| \leq c$. 
 %

To this end, we 
fix  a $G$-equivariant closed embedding 
$X \hookrightarrow \mathbf{A}^N$ into a linear $G$-representation
$\mathbf{A}^N$ where $\Gm$ acts on the various coordinate by powers $\lambda^i$ for $1 \leq i \leq d$.\footnote{
To produce this, we decompose the coordinate ring $\F[X]$ according to $Z$-weights; then $\F[X]$ is generated by homogeneous generators of (positive) bounded degree $\leq d$, and we write $V_j = \F[X]_{(j)}$ for $j \leq d$. Note that each $V_j$ is a finite dimensional $G$-representation on which $Z$ acts by $j$th power. The surjective ring homomorphism 
$\otimes_{j=1}^d \, \mathrm{Sym}(V_j^*) \twoheadrightarrow \F[X]$ gives the desired embedding.}
Then $\Omega$ is contained in a compact subset of $\mathbb{A}^N$, which
can be taken to have the form $a \mathcal{O}^N$
for some $a = (a_v) \in \mathbb{A}^{\times}$.
Now $a_v \mathcal{O}_v$ is covered by $\max(1, |a_v|)$ translates of $\mathcal{O}_v$, so
$\Omega'$   is covered by $\|a\|^N$ translates of $\mathcal{O}^N$,
and therefore, by the fact that $\# (F \cap \mathcal{O}) = q$ and a pigeonholing argument, 
$$ \# (F^N \cap \Omega')  \begin{cases} \leq  q^N \prod_{v} \max(1, |a_v|), & \textrm{any $a$} \\ 
 =0,  & \prod_{v} |a_v| < 1. \end{cases} $$
 
 By the same reasoning for any $\lambda \in F_w^{\times}$,
the size of the set on the left-hand side of \eqref{processed} is bounded by a constant multiple of
$ \max(1, |a_w| |\lambda_w|^{-1})^N$
always, and is zero for $|\lambda|$ sufficiently small.  
\qed

%
%
 \subsection{Whittaker normalization} \label{Whittaker normalization}
  The paper \cite{BZSV} eschews Whittaker normalization, but
  since we will be studying integrals that unfold to the Whittaker model
  it is very convenient to use it.  Let $f$ be an unramified automorphic form on $G$. 
  Let $W_f$ denote the Whittaker period of $f$ with probability normalization:
  \begin{equation} \label{Whitdef} W_f(g) = \int_{u \in [U]} \psi(u) f(ug) d^{\psi}u.\end{equation}
where, following the notation of \S \ref{groupmeasures},
 the measure $d^{\psi} u$ on $U(\adele)$ is normalized to give mass one to the adelic quotient.

Of particular importance to us is the value $W$ at the point 
\begin{equation} \label{a0ex} a_0 := e^{2 \check{\rho}}(\partial^{-1/2}) \in G(\mathbb{A})
 \mbox{ --- e.g. for $\GL_2$ } 
  a_0 =   \left( 
  		\begin{array}{cc} 
			\partial^{-1/2} & 0 \\ 
			0 &   \partial^{1/2} 
			\end{array} \right).\end{equation}
 which is 
 ``the most antidominant point in the torus at which $W_f$ is nonvanishing.''
%
 We denote this value of $W$ by the symbol $W_f^0$: 
\begin{equation}
\label{a0def}  W_f^0 := W_f(a_0).\end{equation} 
 We say that $f$ is {\em Whittaker normalized} if  
 \begin{equation}
 \label{Whitnormnew}
 W_f^0 = \Delta^{\langle \rho, \check{\rho} \rangle - \dim(U)/4}.
 \end{equation}
 where $\Delta=q^{2g-2}$ as in \eqref{Deltadef}; 
 or equivalently 
\begin{equation} \label{Whitnorm} W_f^0 = q^{-\beta_{\mathrm{Whitt}}/2} \, \text{ where } \beta_{\mathrm{Whitt}} = (g-1)(\mathrm{dim}(U)-\langle 2\rho, 2\check{\rho})\rangle\end{equation} 
For example, for $\mathrm{SL}_2$ we have $\beta_{\mathrm{Whitt}} = -(g-1)$
 and for $\mathrm{SL}_3$ we have $\beta_{\mathrm{Whitt}}(g-1) \times (3-8) = -5(g-1).$ 
 Motivation for this is given in \cite{BZSV}. 
 From the point of view of \cite{BZSV}, Whittaker normalization is a special
 case of the period conjectures rather than something to be distinguished. However, for our current purposes, it is most economical to 
 distinguish the role of the Whittaker model. 
%
%
%
%
%

\subsubsection{The Casselman--Shalika formula}

For later use we observe then that  
for $a=(a_v)$ in the adelic torus we have
\begin{equation}  \label{shifted CS} W_f(a) = W_f^0 \prod_{v} W_v^{\mathrm{un}}(a_v a_{0,v}^{-1})\end{equation}
where $W_v^{\mathrm{un}}$ is the standard normalization of the unramified Whittaker function taken with respect
to an unramified additive character, by which we mean
\begin{equation} \label{Wvun} W_v^{\mathrm{un}}:   e^\chi(\varpi_v) \mapsto  q_v^{-\langle   \chi , \rho \rangle}   s_{\chi}(A_v) \mbox{
for $\chi \in X_*(A)$}\end{equation}
with
\begin{itemize}
\item  $A_v \in \check{G}(\kk)$ the Satake parameter of $f$ at $v$;
\item $s_{\chi}(A_v)$ is the trace of $A_v$ in the representation of $\check{G}$ with highest weight $\chi$, assuming $\chi$ to be dominant;
otherwise we understand it to be zero.
\end{itemize}

For later use we note the consequence of \eqref{Wvun}:
\begin{equation} \label{topolish} W_v^{\mathrm{un}}(e^{\chi}(\varpi_v) a_{0,v}^{-1})
 = q^{-\langle \chi+2 m_v \check{\rho}  , \rho  \rangle } s_{\chi + 2m_v\check{\rho}}(A_v). \end{equation}
 with $m_v$ as in \eqref{partialdef}.

 \section{Spectral periods} \label{spectralperiods} 
 
 \subsection{Motivation}
  
 A basic principle enunciated in \cite{BZSV} is that one should index symmetrically
 the automorphic and Galois side of periods {\em by the same data.} Thus, in place
 of an $L$-function, which is attached to a linear representation of $\check{G}$,
 one should consider, more generally, $\check{G}$-spaces $\check{X}$
 (or even more broadly, Hamiltonian actions, but we will not need this). 
 These nonlinear $L$-functions, which
at first may seem a rather exotic concept,  arise naturally in the Rankin--Selberg method, as we already mentioned in \S \ref{RS0}. We now 
set up these notions more carefully. 

In the current paper, we have used a fairly direct definition of the spectral period as a certain Euler product.  Although not apparent, this definition
is compatible, up to normalization issues, with that used in \cite{BZSV} -- the relationship is discussed in 
\S \ref{QuillenSection}.

 \subsection{The spectral period attached to $\check{X}$ (conical case)} \label{conical}

 We follow the notation set up in \S \ref{autperiods}. 
 In the current section \S \ref{conical} we assume throughout that
 \begin{center}
 $\check{X}$ is a conical $\check{G}$-space (as defined in \S \ref{Xdef})
 \end{center}
 and will discuss the general case in \S \ref{nonconical}. 
 Let  $\varphi: \Gamma_F \to \check{G}(\kk)$
be a Galois parameter that fixes only the origin $0 \in \check{X}$. 

%
%
 \subsubsection{Local nonlinear $L$-functions (conical case)}  \label{nonlinearLdef} 
 We define the {\em local $L$-function attached to $\check{X}$} at a place $v$ to be:
\begin{equation} \label{eq: nonlinearLdef} L(\check{X}, \varphi_v,s) :=
\mbox{ trace of $\mathrm{Frob}_v \times q_v^{-s}$ on $\kk[\check{X}]$}.\end{equation} 
  Here,
  we regard $\mathrm{Frob}_v \times q_v^{-s}$
  as belonging to the $\kk$-points of
  $G \times \Ggr$; the action of these $\kk$-points is by pullback via the right action, $\lambda \cdot f (x) = f(x\lambda)$.
On the right hand side, we have a trace on an infinite-dimensional vector space,
which we understand to be the limit  of the corresponding finite-dimensional traces
taken on functions of bounded degree. 
 We will see in a moment that the limit exists
so long as the real part of $s$ is sufficiently large. 
We first verify that this definition generalizes the familiar example of local $L$-factors.  

 \begin{lem} \label{linearL}
Suppose $\check{X} = \mathbf{A}^n$
and the action is linear, arising from a representation
$\varphi: \check{G} \rightarrow \GL_n$;    then
the local $L$-function $L(\check{X}, \varphi_v,s)$
  coincides with the standard
  local $L$-factor  attached to the $n$-dimensional representation $\varphi$. 
   \end{lem}
  \proof 
We compute: 
\begin{equation} \label{basic defn}\sum_{n}  q_v^{-ks} \mathrm{trace}(\mathrm{Frob}_v| \mathrm{Sym}^k \mathbf{A}^n)=
\det(1-q^{-s} \mathrm{Frob}_v)^{-1}\end{equation} as required. 
We emphasize here that $\check{G}$ acts on the right on the ambient $\mathbf{A}^n$, but the associated
action on functions, defined by $g \cdot f(x) = f(xg)$, corresponds to symmetric powers of the standard representation
where $\GL_n$ acts on the left. 
 \qed

Regarding the convergence of these nonlinear local $L$-functions, we decompose the coordinate ring of $\check{X}$ into $\Ggr$-weights, in which only nonnegative weights appear by our definition of conical. Then there exists an integer $d \geq 0$ such that $\kk[\check{X}]$ is generated in degrees $\leq d$, and each homogeneous piece $\kk[\check{X}]_{(j)}$ for $1 \leq j \leq d$ is a finite dimensional $\check{G}$-representation. Then we see that $L(\check{X}, \varphi_v, s)$ is dominated by 
a function of the form 
$$\prod_{1 \leq j \leq d} \prod_{\lambda \in \Lambda_j} \frac{1}{1-q_v^{-j\sigma} \lambda}$$
where $\sigma$ is the real part of $s$, and $\Lambda_j$ is the set of absolute values of eigenvalues of $\mathrm{Frob}_v$ on $\kk[\check{X}]_{(j)}$. In particular, the limit defining $L(\check{X},\varphi_v,s)$ always exists when $\mathrm{Re}(s)$ is large enough.

\subsubsection{Global nonlinear $L$-functions (conical case)} \label{RS} 
We now define the {\em nonlinear $L$-function attached to $\check{X}$}
via
 \begin{equation} \label{LXdef0prod}
L(\check{X}, \varphi, s) =  \prod_v L(\check{X}, \varphi_v, s)\end{equation} as an Euler product. In general it will not have
 meromorphic continuation in $s$. 
 However, by the discussion after Lemma \ref{linearL}, we see
 that at least it will be absolutely convergent for $\mathrm{Re}(s) \gg 1$ in the conical case.

In particular, if $f$ is an everywhere unramified automorphic form on $G$, 
with Galois parameter $\varphi: \Gamma_F \rightarrow \check{G}(\kk)$, 
then $L(\check{X}, \varphi, s)$ is defined as a meromorphic function of $\mathrm{Re}(s)$ 
sufficiently large. Moreover, {\em this meromorphic function can be defined in terms of $f$ alone}, i.e.
whether or not one knows that the Galois parameter $\varphi$ exists: the product in \eqref{LXdef0prod}
refers only to the various Satake parameters $\mathrm{Frob}_v \in \check{G}(\kk)$, which are of course determined by the Hecke eigenvalues of $f$. 
This pleasant state of affairs is special to the case of $\check{X}$ conical;
in general, as discussed in \S \ref{nonconical},  we really need access to the Galois parameter in order
to define a nonlinear $L$-function.
  
  \subsubsection{The spectral period (conical case)}
Of particular importance to us will be the nonlinear $L$-function
evaluated at $\frac{1}{2}$ -- at least when this makes sense, that is to say,  when it can be meromorphically continued
to a function defined at that point.

We define
the {\em spectral period} attached to $\check{X}$ and the automorphic form $f$  via the rule 
\begin{equation} \label{LXdef0}  L_{\check{X}}(\varphi) :=
\left[ \mathfrak{z} 
\Delta^{\frac{\varepsilon -\dim \check{X}}{4}} \right] \times  L(\check{X}, \varphi, \frac{1}{2}), \end{equation}
where 
\begin{equation} \label{zdef}  \mathfrak{z} = \mbox{
  the scalar by which 
 $\check{\eta}(\partial^{-1/2}) \in G(\mathbb{A})$
acts on $f$.}
\end{equation}
 Here
$\check{\eta}: \check{G} \rightarrow \Gm$ is the eigenmeasure for $\check{X}$ and $\varepsilon$ is the $\Ggr$-weight on the eigenmeasure, as defined in \eqref{etadef}. Note that 
$\check{\eta}$ dualizes to a  central cocharacter of $G$ denoted  in \eqref{LXdef0} by the same letter. 

We give this the grandiose name of  ``spectral period''
because (for reasons not to be recalled here, but disucssed in \cite{BZSV}) it 
in fact is very closely analogous to the automorphic period; in particular the normalization factor
arises from considerations described in {\em op. cit.}
In our application, we will  understand it to be defined
only when the corresponding function of $s$ admits a meromorphic continuation to $s=\frac{1}{2}$.

%
 
\subsubsection{When is the global nonlinear $L$-function $L(\check{X}, \varphi, s)$ actually a ratio of $L$-functions?} \label{CIsec}
  In general, $L(\check{X}, \varphi, s)$ is not a ratio of $L$-functions, and we would 
  not expect it to have meromorphic continuation to the full complex plane. However,
  it {\em will} be a ratio of $L$-functions when $\check{X}$ is complete intersection, as we now explain. 
  
Indeed --- although we will not explicitly use it in this paper ---
  Lemma \ref{linearL} can be generalized
  using the tangent complex; this 
  allows us to assign to the point $0 \in \check{X}$
  an infinite sequence $T_i (i \geq 0)$ of vector spaces, 
  the ``cohomology of the tangent complex.''
  Here $T_0$ recovers the usual tangent space to $\check{X}$ at $0$, and for us
  all $T_i \ \ (i \geq 0)$ are finite dimensional.
  Since $0$ is fixed by $\check{G} \times \Ggr$,
  each $T_i$ is a  graded $\check{G}$-representation;  write, accordingly,
  $T_i^{(d)}$ for the $d$th graded piece, which is nonzero only for $d \geq 1$. 
  Then one can prove (cf. \S \ref{QuillenSection}) the following expression for the $L$-factor in
  terms of standard ones:  
\begin{equation} \label{Tsf} L_v(\check{X}, \varphi,s) = \prod_{i \geq 0, d \geq 1}  L_v(T_i^{(d)}, \varphi,   ds)^{(-1)^i}.\end{equation} 
  understood as a function of meromorphic functions in
a right half-plane. 
  Now:
  \begin{itemize}
  \item $\check{X}$ is smooth at $0$ if and only if $T_i$ vanishes
  for all $i > 0$. In this case, \eqref{Tsf} recovers Lemma \ref{linearL}. 
  \item $\check{X}$ is a {\em locally complete intersection} at $0$ 
  if and only if $T_i$ vanishes for $i >1$ (\cite{Quillen} Theorem 5.4(iv)).
  In this case, \eqref{Tsf}
  expresses the nonlinear $L$-function as a ratio of standard $L$-functions: 
    $$ L(\check{X}, \varphi,s) =  \frac{ L(T_0, \varphi,  s)}{  \prod_{d} L( T_1^{(d)}, \varphi,  ds)} .$$ 
A typical case in which this occurs is when $\check{X}$ is defined inside a vector space $V$ 
 by the common zero-locus of $G$-invariant homogeneous polynomials $P_1, \dots, P_t$ of degrees $d_1, \dots, d_t$,
 all assumed larger than $1$, 
 and the $P_i$ form a regular sequence. 
  In this case,  $\check{X}$ is of dimension $\dim(V)-t$, 
$T_0 = V$, and $T_1$ is $t$-dimensional, with generators in degree $d_1, \dots, d_t$, and we get
\begin{equation} \label{LCI formula} L(\check{X}, \varphi,s) = \frac{L(V, \varphi, s)}{\prod_{i=1}^t \zeta(d_i s)}.\end{equation} 
This interpretation of the right hand side as a nonlinear $L$-function attached to a cone
was a key realization in developing the ideas of this paper. See also the motivating example in \S \ref{RS0}.
\item 
  In the general case it is  usually the case that infinitely many $T_i$ are nonvanishing. In this case, the corresponding
Euler product coincides with an infinite product of $L$-functions and will have an essential boundary. This is the Estermann phenomenon, well-known in analytic number theory.
\end{itemize}

%
%
 
 \subsection{The nonconical case} \label{nonconical} 
Although
not necessary for the main results of this paper,
the general duality formalism applies a similar definition to other $\check{X}$,
although it requires explicit access to a 
Galois parameter
$\varphi: \Gamma \rightarrow \check{G}(\kk)$
and not merely the Satake parameters. 
 
 In this case, 
 $\Gamma$ acts on $\check{X}$ through $\varphi$ and
we will assume that the fixed point set is discrete, otherwise
we regard the definition as invalid.
 
 Let $x_0 \in \check{X}(\kk)$ be an arbitrary fixed point of the $\Gamma$-action.
 It is then also fixed by $\Ggr$.  
  We may then define the $L$-function attached to $(\check{X}, x_0)$ to be
 ``the nonlinear $L$-function attached to the linearization of $\check{X}$ at $x_0$.''
 Indeed, our previous discussion in the conical case depended only on the germ of $\check{X}$ at $0$,
 and when reformulated in those terms applies similarly here. 
 The result is the following definition: 
\begin{equation} \label{CLR} L_{x_0}(\check{X},\varphi, s) := \prod_{v}    \mathrm{trace}\left( \mathrm{Frob}_v \times q_v^{-s} | \widehat{\mathcal{O}}_{\check{X}, x_0}\right)\end{equation} 
 where the trace on the completed local ring $\widehat{\mathcal{O}}$ is defined as the sum of traces over the finite dimensional spaces
$ \frac{ \mathfrak{m}^n}{\mathfrak{m}^{n+1}}$ with $\mathfrak{m}$ the maximal ideal. Each factor defines a rational function of $q_v^{-s}$; 
for the reasons recalled in \S \ref{CIsec} the product defines a rational function of $q^{-s}$ if $\check{X}$ is complete intersection at $x_0$;
otherwise it may not have an analytic continuation in $q^{-s}$.   

We then define the 
{\em normalized} spectral period by taking the value at $s=1/2$, if defined, and summing over fixed points (cf. \eqref{LXdef0}): 
 \begin{equation} \label{LXdef} L_{\check{X}}(\varphi) := \mathfrak{z} 
\Delta^{\frac{\varepsilon -\dim \check{X}}{4}}  \sum_{x \in \mathrm{Fix}(\varphi,X)}  
L_x(\check{X},  \varphi, \frac{1}{2} )  \end{equation}
Here $\mathfrak{z}$ is as in the conical case and $\varepsilon \in \ZZ$ is the $\Ggr$-weight on the eigenmeasure as in \eqref{etadef}. We can equivalently express $\mathfrak{z}$ as the value of   $\check{\eta} \circ \varphi: \Gamma_F \rightarrow \kk^{\times}$ at $\partial^{-1/2}$, viewed as an element of $\Gamma^{\mathrm{ab}}$ via global class field theory, which is compatible with the prior definition in the conical case \eqref{LXdef}. 

\section{Weak duality}

Let $(G_1, G_2)$ be split reductive groups 
in duality. (We avoid the usual notation $G$ and $\check{G}$ to better bring out the symmetric nature of the situation.)
We shall say that $G_i$-varieties $X_i$
are   {\em \wdd} if, informally speaking, we have the following equalities of periods evaluated on cusp forms
$$\mbox{automorphic $X_1$-period} \sim \mbox{spectral $X_2$-period} $$
$$\mbox{automorphic $X_2$-period} \sim \mbox{spectral $X_1$-period} $$
where $\sim$ means agreement up to a half-integer power of $q$. 
The formal definition reads as follows: 

\begin{definition} \label{weakly dual def}
Follow notation as in \S \ref{notation}. 
We say that the $G_i$-varieties $X_i$ (defined over $\Q$) are \wdd  if, 
for a suitable $N$, there are 
  integral models $\mathcal{X}_i$ of $X_i$ over $\Z[1/N]$, \footnote{The formal definition of this does not matter much, but one could take
  ``a flat scheme over $\Z[1/N]$ equipped with an action of the split form of $G_i$.''}
 such that, for any finite field $\F$ of size $q$ with characteristic not dividing $N$,
and any projective smooth curve $\Sigma$ over $\F$ with discriminant $\Delta$, we have equalities:  
\begin{equation}
\label{weakduality}
P_{X_1}(f_1) = \Delta^{a_{12}/4} L_{X_2}(\varphi_1) \mbox{ and }P_{X_2}(f_2) = \Delta^{a_{21}/4} L_{X_1}(\varphi_2) \end{equation}
whenever: 
\begin{itemize}
\item $X_i = \mathcal{X}_i \otimes \F$; 
\item $f_1$ (resp. $f_2$) is any Whittaker-normalized as in \eqref{Whitnorm}, everywhere unramified,
cuspidal, tempered automorphic form on $G_1$ (resp. $G_2$)
with Langlands parameter $\varphi_1$ valued in $G_2$ (resp. $G_1$). 

\item $P_{X_i}(f_i)$ are pairings of the $f_i$ with the Poincar{\'e} series for $X_i$, as defined in
in \S \ref{autperiods} 
\item $L_{X_2}(\varphi_1)$ and $L_{X_1}(\varphi_2)$ are nonlinear $L$-functions, defined
as in \S \ref{conical} and \S \ref{nonconical}. 
\item the equalities are understood to hold whenever 
the fixed locus of $\varphi_i$ on the $X_i$ are discrete. In particular, if this fixed locus is empty,
the period must vanish.
\item $a_{12}, a_{21} \in \Z$ are integers, which we term the {\em discrepancy} of the pair $(X_1, X_2)$.

\end{itemize}
\end{definition} 

\begin{rem} In fact, \cite{BZSV} formulates this when the fixed locus on $T^*X$ is discrete -- a stronger condition.
\end{rem}

The proposal of \cite{BZSV} amounts to a large class of
pairs $(X_1, X_2)$, both smooth, which are \wdd and for which
  the discrepancies $a_{12}=a_{21}$
vanish identically.   For example, the examples in the following theorem all fall under the umbrella
of the hyperspherical duality proposed in \cite{BZSV}: 

\begin{thm}\label{thm: smooth}
Each of the pairs in the table below is, in the sense of Definition
\ref{weakly dual def},   \wdd  with discrepancy zero. \footnote{In the third row, i.e., the Godement--Jacquet and Rankin--Selberg example, one has central tori acting trivially on the relevant varieties. In this situation a regularization is needed, which we explain in Section \ref{sect: cancelZetaOnes}}
\end{thm}

 Here the $\Ggr$ action acts by scaling on each vectorial component of $X$ and otherwise acts trivially.

\begin{table}[!hbtp]
\begin{tabular}{|c|c|c|c|c|c|}
\hline
 	& $G_1$		& $X_1$  				&  $G_2$ 						&  $X_2$ 				&   \\ \hline
 I.--T. \S \ref{Tate} &  	$\Gm$	&  	$\mathbf{A}^1$	& $\Gm$  						& $\mathbf{A}^1$		&  I.--T. \S \ref{Tate} \\ \hline
group \S \ref{gp} &  	$\GL_n$	&  	$\GL_n^2$	& $\GL_n$  						&$\GL_n^2$		&  (twisted) \S \ref{gp} \\ \hline

G.--J. \S \ref{GJcase} &  	$\GL_n^2$	&  $\mathrm{Mat}_n$& $\GL_n \times \mathbf{A}^n$  	& $\GL_n^2$		&  J.--PS.--S. \S \ref{JPSS}\\ \hline
Hecke \S \ref{Hecke} &  	$\PGL_2$		&  $\PGL_2/\Gm$	& $\mathbf{A}^2$  				& $\SL_2$			&  Eisenstein \S \ref{Eiscase} \\ \hline

\end{tabular}
\end{table}
We will give the proof of this theorem in Appendix \ref{smoothproofs}; there we will also spell out precisely what the actions are (there are some ``transposes'' and ``inverses'' at some points
that are perhaps not obvious).   These are all very standard computations. Our goal in computing them is to simply spell out
how all the constants (i.e. the powers of $q$) match up to give discrepancy zero.
This relies crucially on the precise  shifts and normalizing factors that have been embedded in the statement of duality.
By contrast, our singular examples will {\em not} have discrepancy zero.

\subsection{Convergence issues in the reductive case}\label{sect: cancelZetaOnes}

If $G$ contains a central torus acting trivially on $X$
the integral will diverge.  Nonetheless,
there are plenty of  \wdd cases that involve this situation,
where the above equality holds if we allow ourselves to ``formally cancel''
a divergent factor involving $\zeta(1)$ on both sides; on the automorphic
side, this factor arises from the volume of $[\Gm]$, 
and on the spectral side, it arises from a $\zeta(1)$ factor. 
Although by far the most convenient way to think
about this situation is in terms of such ``formal cancellation,''
 we make the following definition for concreteness.
 
 In the situtation where a nontrivial central torus of $G_i$ acts
 trivially on $X_i$, 
we say that $(G_1, X_1)$ and $(G_2, X_2)$ are \wdd 
if they are ``induced'' from a  \wdd pair
$$(\bar{G}_i, Y_i),$$ 
with $\bar{G}_i$
a ``subquotient'' of $G_i$ that has no central torus acting trivially on $Y_i$.

The meaning of ``subquotient'' and ``induced'' is as follows:
We start with central tori $Z_i \subset T_i$
which are orthogonal to one another.
These define saturated sublattices $X_*(Z_1) \subset X_*(T_1)$ and $X_*(Z_2) \subset X_*(T_2)$.
The free abelian groups $X_*(T_i)$ are in duality, and we can therefore form
the orthogonal complements $X_*(Z_1)^{\perp} \subset X_*(T_2)$ and vice versa;
these correspond to 
 ``orthogonal'' tori $Z_i^{*} \subset T_i$. 
 
These orthogonal tori correspond to reductive subgroups
$G_i^* \subset G_i$,  e.g. $G_2^*$
is the subgroup associated by duality to $G_1/Z_1$; 
and we take
$$\bar{G}_i = G^*_i/Z_i.$$
The meaning of ``induction'' is that 
we pass from a $\bar{G}_i$-space $Y_i$
to the $G_i$-space $G_i \times_{G_i^*} Y_i$.

\subsection{The duality conjecture of \cite{BZSV}, and why  we say ``weak'' duality here.}
 %
%

We should emphasize that the notion of ``numerical weak duality''  is anticipated to be a strictly weaker notion than the duality of \cite{BZSV}. 

The adjective ``numerical'' refers to the fact that we are imposing constraints only on numerical periods, whereas the conjectures of \cite{BZSV} apply at a geometric level. 
The adjective ``weak'' refers to the (more important) fact that we impose constraints only for {\em cuspidal} periods; by contrast the duality of \cite{BZSV} entails rather precise constraints on Eisenstein periods, which amounts to correspondingly precise constraints on the geometry of lower-dimensional orbits. 

To be slightly more precise, we note the following pattern in the examples we study. Pairing $\theta_X$ with cusp forms only probes the
geometry of the open $G$-orbit in $X$, in the sense that smaller-dimensional orbits will not contribute to the automorphic period. However, the automorphic period does not depend {\em solely} on the open orbit, because
the Schwartz function on that orbit depends on the global geometry of $X$. Dually, considering only the spectral $\check{X}$-period of cuspidal $L$-parameters, only a unique fixed point on $\check{X}$ will contribute; \textit{a priori} one can imagine that for $L$-parameters with small image one has to consider more fixed points, or fixed subvarieties on $\check{X}$. However, the spectral period also does not depend {\em solely} on this fixed point, since the completed local ring at that point depends on $\check{X}$. 

 On the other hand, it seems of value to make this purely numerical definition, because
  it is easy to define and easy to compute with. It is reasonable to expect
 that many \wdd pairs are not far from being a dual pair of boundary conditions
 in the stronger sense of \cite{BZSV}.

\newcommand{\Ad}{\mathrm{Ad}}
\newcommand{\Spec}{\mathrm{Spec}}
\newcommand{\OO}{\mathcal{O}}

\section{Geometry of  $U^0 \backslash G$ and $T^0 U^0 \backslash G$}

\subsection{Setup} \label{setup0}
The following general setup will be used in the remaining sections of the paper. 
We advise the reader to look forward at the examples in \S \ref{Examples}
before trying to absorb the abstraction below.  It is worth noting
that the results of this section can also be proved more explicitly and readily ``by hand''
in each case; we have adopted to treat the two cases uniformly at some cost of extra abstraction.
The advantage of uniformity is not just to shorten proofs: it makes clear how the computation
of integral points on an affine closure $\overline{S \backslash G}$ is actually controlled by the root data.

We fix a coefficient field $k$. For our  later applications we will want to take $k$ to mean either the ``automorphic'' or the ``Galois'' coefficient field (see \S \ref{Fandk}). 
\footnote{It might be more uniform to work over $\mathbf{Z}$ at the get-go since an integral model of our group actions is needed for the automorphic calculation, but we will avoid doing so at this stage, leaving the choice of integral model to \S \ref{subsect: intModel}.}

Let $G$ be an arbitrary split reductive group over $k$ whose center is $1$-dimensional. Choose a Borel subgroup $B = UT$ where $U$ is its unipotent radical, and $T$ is a maximal split torus in $G$. 
Let $\theta: G/[G,G] \to \Gm$ be an isomorphism, and let $\check{\theta}: \Gm \dashrightarrow Z(G)$ be a rational character such that $\langle \theta, \check{\theta} \rangle = 1$, where we allow ourselves to pass to a finite cover to define this cocharacter. 
  Note that $\theta$ and $\check{\theta}$
are both determined up to the same sign. 

We shall use the usual notation $\check{\rho}$ for the sum of all positive coroots, and $\rho$ for the sum of all positive roots;
note that  $\langle \check{\rho}, \alpha \rangle = 1 $ for all simple roots $\alpha$.  
Define $$\check{\rho}' := \check{\theta}+\check{\rho}: \Gm \dashrightarrow G,
T^0 := \mathrm{image}(\check{\rho}').
$$
where the dashed arrow again means that we allow ourselves to pass to a finite cover of $\Gm$ on which this is defined.

Let $\Psi: U \rightarrow \mathbf{G}_a$ be the unique 
homomorphism restricting to the identity on each pinned root subgroup,
and let $U^0$ be the kernel of $U$; 
we will also denote $\Psi$ as $u \mapsto \bar{u}$ for short. Set $S=T^0 U^0$. 
 Our main concern in this section is the geometry of the space
  \begin{equation}\label{eqn: cone} X  = \overline{\mbox{$S \backslash G$}} = \mbox{ the affine closure of  $S \backslash G$}\end{equation}
  
  Explicitly, $X$ is the affine variety whose ring of functions are $S = T^0 U^0$-invariant functions
  on $k[G]$.  This is a $k$-variety, and will be denoted $X_k$ if there is ever danger of confusion about the field of definition.

    The main concern of this section is to understand, as explicitly as possible,
  which points on $X$ are ``integral'' (for a suitable integral model  to be defined in \S \ref{subsect: intModel}) in the case
  when $k$ is a nonarchimedean local field.  The main result will be Proposition \ref{intpointsX}.

\subsection{Regular functions on $X$}
 Our first goal is to compute the ring of functions on $X$; this
 is accomplished by the following proposition:
 
\begin{prop}
\label{Xfunction}
Let $k$ be a field of characteristic zero. 
 Let $k[X]$ be the ring of regular functions on $X_k$. Then, as $G$-representation, 
 $$k[X] \cong \bigoplus_{\lambda} \,   V_{\lambda}^*$$
where $V_{\lambda}$ is the representation of highest weight $\lambda$
and the sum is taken over all $\lambda \in X^*(T)$ such that 
\begin{equation} \label{relevant}  0 \leq    \langle \lambda, \check{\rho}' \rangle \leq \min_{\check{\alpha} \, \text{ simple}} \langle \lambda, \check{\alpha} \rangle.\end{equation}

\end{prop}
We say that a   weight $\lambda$ is \textit{relevant} if it satisfies these inequalities.
A relevant weight is automatically dominant. 
We will analyze the cone of relevant weights more carefully in \S \ref{relevant}.

\proof (of Proposition \ref{Xfunction})  
We have
\begin{equation} \label{Xiso}  k[X] \simeq \bigoplus_{\lambda} V_{\lambda}^{S} \otimes V_{\lambda}^*.\end{equation}
where we associate to $v_{\lambda} \otimes v_{\lambda}^*$
the function $\langle v_{\lambda} , g  v_{\lambda}^* \rangle$
on $(S\backslash G)_k$. The proof of the Proposition amounts to computing the dimension of $S$-invariant vectors in $V_\lambda$.

Let $w$ be a representative
in $G$ for a long Weyl element, i.e. normalizing $T$ and sending $B$ to its opposite.
By Borel--Weil,  the rule sending
$v \in V_{\lambda}$ to the function $\langle g v, v^* \rangle$ for $v^*$
a highest weight vector  of the dual representation $V_{-w\lambda}$ gives (after 
fixing an invariant pairing $V_{\lambda} \times V_{-w\lambda} \rightarrow k$):
 \begin{equation} \label{fdef} 
V_{\lambda} \simeq \{ f \in k[G]: f(b g) = w\lambda(b) f(g) \}.\end{equation}
where $G$ acts by right translation on the right hand side.

The isomorphism \eqref{fdef} induces an injection
$$ \mathrm{I}:  V_{\lambda}^{U^0} \rightarrow \mbox{polynomials on $\mathbf{G}_a$}$$
sending a function $f$ 
to its restriction to $f(w u)$, which 
descends from $U$ to $U/U^0 \simeq \mathbf{G}_a$. 
The action of $U/U^0$ on the  domain of $\mathrm{I}$
corresponds to translation by $U^0 \backslash U \simeq \mathbf{G}_a$ on the codomain.
We claim that:
\begin{quote} (*)  the image of $\mathrm{I}$  precisely consists of polynomials $P$
 that have degree at most $\langle \lambda, \check{\alpha} \rangle$ for all simple roots $\alpha$. 
 \end{quote}

 Any polynomial $P$ on $\mathbf{G}_a$ determines a function $f_P$ on the big cell $B w B$ satisfying \eqref{fdef},
and to check (*) we must determine when $f_P$ extends to $G$;
for this it is sufficient to check that it extends over all codimension one
$B$-orbits, which reduces us to  a similar question for $\SL_2$, 
by means of pullback along the $\SL_2$-embeddings associated to simple roots.
Indeed, the pullback of $f_P$ along such an embedding defines a function on the big cell for $\SL_2$
satisfying still

$$\bigg\{f \in k[\mathrm{SL}_2] \, : \, f\bigg(\begin{bmatrix} t&x\\0&t^{-1}\end{bmatrix}g\bigg) = t^{-d} f(g)\bigg\}$$
where $d =\langle \lambda, \check\alpha \rangle$, with $\alpha$ is the simple root defining the $\SL_2$-embedding. 
Consider $f$ whose restriction  to the unipotent radical  $U \cong \Ga$ is in standard coordinates a monomial of degree $n$, i.e., 
$$f_P \bigg(w\begin{bmatrix} 1&x\\0&1\end{bmatrix}\bigg) = x^n$$

Then we have
$$f_P\bigg(\begin{bmatrix} t& y\\0&t^{-1}\end{bmatrix} \begin{bmatrix} 0&1\\-1&0\end{bmatrix} \begin{bmatrix} 1&x\\0&1\end{bmatrix} \bigg) = f_P \bigg(\begin{bmatrix}-y & -xy+t\\-t^{-1} & -t^{-1}x\end{bmatrix}\bigg) = t^{-d}x^n$$
In other words, on the big cell $f_P$ is given by
$$f_P\bigg(\begin{bmatrix} x&y\\z&w\end{bmatrix}\bigg) = z^{d-n}w^n$$
which extends to all of $\mathrm{SL}_2$ if and only if $d \geq n$, so we have verified our claim (*). 

Thus $\mathrm{I}$ identifies $V_{\lambda}^{U^0}$ with polynomials of degree at most $\min \langle \lambda, \check{\alpha} \rangle$.  We claim
that the image of $T$-invariant vectors
can only consist of multiples of the monomial $x^{\langle  \check{\rho}', \lambda \rangle}$.
 Indeed, in order that $v \in V_{\lambda}^{U^0}$ be fixed by $T^0$,
the corresponding function $f_v$  (with respect to \eqref{fdef}) must satisfy
$f_v(w u t) = f_v(w  u)$; the left
hand side equals  
 $f_v(t^{w}w  t^{-1} u t) =  t^{\lambda} f(w \Ad(t^{-1}) u)$.  
Now,  taking $t= x^{{\check{\rho}'}}$, 
$\Ad(t^{-1})$ effects a dilation by $x^{-1}$
 on $U/U^0 \simeq \mathbf{G}_a$,  and $t^{\lambda}= x^{ \langle \lambda, \check{\rho}' \rangle}$.
This proves our claim, and concludes the proof of the Proposition. 
 In fact, we have proved that there exists a unique function
  $\tilde{P}_{\lambda}$ on $G$ such that
  \begin{equation} \label{Plambdadef} \tilde{P}_{\lambda}(bwu) =  b^{w \lambda} \bar{u}^{\langle \check{\rho}', \lambda \rangle},\end{equation} 
and it corresponds under \eqref{fdef} to the unique line in $V_{\lambda}^S$.
  \qed

 \subsection{The cone of relevant weights}  \label{relevant section}
 The set defined by \eqref{relevant} defines a convex cone inside $X^*(T) \otimes \mathbf{R}$, that is to say, a closed convex set
invariant by scaling under $\R_{>0}$. 
We will describe it more explicitly with the following notations. 

Write $\rho, \check{\rho}$ in terms of simple roots and coroots: 
\begin{equation} \label{nadef}  \check{\rho} = \sum_{\check{\alpha}} n_{\check{\alpha}} \check{\alpha}, \rho = \sum_{\alpha} n_{\alpha} \alpha.\end{equation}
We further write
\begin{equation} \label{adef} a = \langle \rho, \check{\rho} \rangle = \sum n_{\alpha} = \sum \check{n}_{\alpha}.\end{equation}
where the $\alpha$-sums range over simple roots or simple coroots. 
For the computation of $\langle \rho, \check{\rho}\rangle$ we used that  $\langle \rho, \check{\alpha} \rangle = \langle \check{\rho}, \alpha \rangle =1$
for $\alpha$ simple.

Next we write $\varpi_{\alpha} \in X^*(T) \otimes \R$ 
for the dual basis to the simple coroots $\check{\alpha}$, explicitly:
$\varpi_{\alpha}$ pairs trivially with $\check{\beta}$ for $\beta \neq \alpha$
and with $\check{\theta}$, and pairs to $1$ with $\check{\alpha}$. 
 Note that each element 
$ \varpi_{\alpha} - n_{\check{\alpha}} \theta$ is orthogonal to  
$\check{\rho}'$ and $\langle \varpi_\alpha-n_{\check{\alpha}}\theta, \check{\beta}\rangle = \delta_{\alpha\beta}$, so $ \varpi_{\alpha} - n_{\check{\alpha}} \theta$ belongs to the relevant cone.  Since the element $\rho - (a-1) \theta$  pairs
to $1$ with each simple coroot and with $\check{\rho}'$, we see that it also belongs to the relevant cone.

\begin{lem}  The vectors  named above, that is, 
\begin{equation} \label{extremal relevant} \big\{\varpi_\alpha - n_{\check{\alpha}} \theta: \alpha \text{ simple root}\big\} \cup \big\{\rho-(a-1)\theta\big\}\end{equation}
span the extreme rays of the (real) cone of relevant weights, that it to say, the relevant cone consists of all linear combinations
\begin{equation} \label{basic0} \sum_{\alpha \textrm{ simple}} x_{\alpha}  \left(  \varpi_{\alpha} - n_{\alpha} \theta \right) + y (\rho - (a-1) \theta)\end{equation}
for  nonnegative  reals $x_{\alpha},y$. 
\end{lem}

\proof First of all, note that $\{\varpi_\alpha-n_{\check{\alpha}} \theta: \alpha \text{ simple root}\} \cup \{\rho-(a-1)\theta\}$ forms a basis of $X^*(T) \otimes \R$; indeed, the set $\{\varpi_\alpha - n_{\check{\alpha}} \theta: \alpha \text{ simple root}\}$ is evidently linearly independent, and as noted previously we have
\begin{equation} \label{rrr}\langle\varpi_\alpha - n_{\check{\alpha}}\theta , \check{\rho}'\rangle = 0 \text{ while } \langle \rho- (a-1)\theta, \check{\rho}' \rangle = 1\end{equation}
so the whole set is linearly independent of size equal to $\mathrm{dim}(X^*(T) \otimes \R)$. Thus, for any $\lambda \in X^*(T)$ we can write
$$ \lambda=\sum_{\alpha \textrm{ simple}} x_{\alpha}  \left(  \varpi_{\alpha} - n_{\check{\alpha}} \theta \right) + y (\rho - (a-1) \theta)$$
for some uniquely determined real numbers $x_\alpha$ and $y$.
Suppose now $\lambda$ is a relevant dominant weight; our goal is to show that the above numbers $x_\alpha$ and $y$ are nonnegative. We can directly compute
$$y = \langle \lambda, \check{\rho}'\rangle \geq 0$$
by \eqref{relevant}. Next, we set
$$\lambda^* := \lambda - y (\rho-(a-1) \theta) = \sum_{\alpha \text{ simple}} \, x_\alpha(\varpi_\alpha - n_{\check{\alpha}} \theta)$$ 
and compute
$$x_\alpha = \langle \lambda^*, \check{\alpha}\rangle = \langle \lambda, \check{\alpha} \rangle -y \langle \rho-(a-1)\theta, \check{\alpha}\rangle = \langle \lambda, \check{\alpha}\rangle - y \geq 0$$
again by \eqref{relevant}.

\qed

   \subsection{An integral model of $X$}  \label{subsect: intModel}
       Our next task is to
       define an integral model of $X$. We will proceed rather crudely -- we do not try to make a ``best'' model
       but just make a class of them (in our examples of interest, there is probably a ``best''
       one).

       We will prove in sequence the following statements,
       working over $k=\Q$.
 Note that the inclusion $\Q[X] \rightarrow \Q[G]$ gives a map $G \rightarrow X$, 
which factors through an open immersion
of $S \backslash G$ into $X$.

\begin{itemize}
\item[(i)] $X_{\Q}$ is normal.

\item[(ii)] The open orbit inside $X_{\Q}$ has codimension $\geq 2$. 

\item[(iii)] Let $I$ be any finite subset of relevant weights that generate all relevant weights under addition. For each $i \in I$ fix a nonzero vector $v \in V_i^S$ (notation as in Proposition \ref{Xfunction}: $V_i$
the representation with highest weight $i$). 
Then the orbit map  of $v_I := (v_i)_{i \in I}$, i.e.
$g \mapsto g^{-1} v_I$, which defines a morphism
$$ G \rightarrow  V_I := \bigoplus_{i \in I} V_i$$
and thereby a map from $G/S$ to the orbit $G.v_I$ of $v_I$, in fact
extends to an isomorphism of $X_{\Q}$ with the closure of $G. v_I$ inside $V_I$.

 \item[(iv)]  Fix $I$ as in (iii) and fix an  integral model for $V_I$.
 For concreteness, for each $V_{\lambda}$ with $\lambda \in I$, we take  the
 integral
lattice inside the $\Q$-representation with highest weight $i$
defined by the $\Z$-analogue of \eqref{fdef}.\footnote{The exact choice does not matter, as we permit
ourselves in the end to discard finitely many primes.}
Set
$$ \mathcal{X} := \mbox{ Zariski closure of $X_{\Q}$ inside this integral model},
$$
where $X_{\Q}$
 is identified to a subvariety of  $V_I$ as in (iii).  
 Then for almost all $p$ the orbit map
 of $v_I$, i.e. the map $(S \backslash G)_{\mathbf{F}_p} \rightarrow \mathcal{X}/p$
 induced by $g \mapsto g v_I$, 
  induces an isomorphism of the affine closure $X_{\F_p}$ of $(S \backslash G)_{\mathbf{F}_p} $, with $\mathcal{X}/p$.
\end{itemize}

 \begin{rem}
 In fact, the only property of $S$ used is that it is a multiplicity one $\Q$-subgroup;
 the above facts go through for a variety of the form $\spec \QQ[G]^S$.
 \end{rem}

We remark that (iv) will not be used, in particular the large residual characteristic assumption will not need to be imposed, until the automorphic calculation in \S \ref{sect: Unfolding}: indeed, the main integrality result of this section (Proposition \ref{intpointsX}) has no characteristic assumptions. However, Proposition \ref{intpointsX} gives a criterion for integrality only for
points of $X$ of a certain type, and  the characteristic assumption comes in later to ensure us that these special points are the only ones that enter into the automorphic calculation.
 \proof 
For the proof of (i) we note that the inclusion $\Q[X] \rightarrow \Q[G]$
induces an inclusion of function fields $\Q(X) \hookrightarrow \Q(G)$. 
An element $f \in \Q(X)$, integral over $\Q[X]$ is {\em a fortiori}
integral over $\Q[G]$, and so belongs to $\Q[G]$; being invariant,
it also belongs to $\Q[X]$. 
 
(ii) follows  via (i) and the following fact in commutative algebra:
let $A$ be a normal domain, 
 and $Z \subset \Spec(A)$ an irreducible closed subset of codimension $1$;
then there are regular functions on $\Spec(A)-Z$ that do not extend to $\Spec(A)$.
We apply the fact with $A=\Q[G]^S$ the ring of functions on $X$,  and $Z$  any irreducible
component of complement to the open $G$-orbit;
if $Z$ had codimension $1$ then there would be regular functions on the open $G$-orbit
that do not extend to $X_{\Q}$, but all regular functions on the open $G$-orbit 
give $S$-invariant functions on $G$ which, by definition of $X$, in fact do extend.

 To check the fact of commutative algebra,   we may replace $Z$ by an irreducible component, and therefore will assume that $Z$ is the subset defined by a single height one prime ideal $\mathfrak{p}$. 
Then $A_{\mathfrak{p}}$ is a discrete valuation ring, and so there exists $z$ in the quotient field
of $A_{\mathfrak{p}}$ with
$z \mathfrak{p} A_{\mathfrak{p}} = A_{\mathfrak{p}}.$
This means, in particular, that $z \mathfrak{p} \subset A \sigma^{-1}$
for some $\sigma \notin \mathfrak{p}$; thus
$$ (z \sigma) \mathfrak{p} \subset A.$$
The element $(z \sigma)$ is an element of the quotient field of $A$;
it doesn't lie in $A$ itself, for otherwise 
we would have $z \in A_{\mathfrak{p}}$. The above inclusion implies that 
  it is regular at all points of $\mathrm{Spec}(A)$ not lying in the closed set defined by $\mathfrak{p}$.

We proceed to the proof of (iii). First of all,  we establish the following two easy facts: 
\begin{itemize}
\item[(a)] Denote $\Q[X]_{\lambda}$ the subspace of $\Q[X]$ transforming according
to the $G$-representation with highest weight $\lambda$;  it is at most one-dimensional. Then
$$ \Q[X]_{\lambda} \Q[X]_{\mu} \subset \sum_{\nu \leq \lambda + \mu} \Q[X]_{\nu}$$
where on the right we order relevant $\lambda$ in the standard way (e.g.
order by the pairing with $\rho$), and
\item[(b)]   the map
$\Q[X]_{\lambda}  \otimes \Q[X]_{\mu} \rightarrow \Q[X]_{\lambda+\mu}$
obtaining by projecting the product to the $\lambda+\mu$ factor is actually {\em surjective}.
\end{itemize}
 Here, (a) follows from the similar assertion concerning decomposition of $G$-representations under tensor product, 
 and (b) follows from the fact that, with reference
to the unique-up-to-scalars nontrivial map $V_{\lambda} \otimes V_{\mu} \rightarrow V_{\lambda+\mu}$,
the image of $V_{\lambda}^S \otimes V_{\mu}^S$ is a {\em nonzero} $S$-invariant vector;
indeed, this unique map is simply multiplication in the Borel-Weil model,
and thus no pure tensor lies in its kernel. 

Continuing to the proof of (iii): 
By definition $S$ is a subgroup of $\mathrm{Stab}_G(v_I)$; on the other hand, since $I$ generates the relevant weights, by fact (b) we see that all $S$-invariant functions on $G$
are actually invariant under
 $\mathrm{Stab}_G(v_I)$; therefore $S = \mathrm{Stab}_G(v_I)$. 
 The induced map $\Q[\overline{G. v_I}] \rightarrow
\Q[G . v_I] \simeq  \Q[X]$ 
is injective because $G. v_I $ is dense in $\overline{G. v_I}$. 

Write   
$\Q[X] = \bigoplus \Q[X]_{\lambda},$
where the sum is taken over relevant weights and $\Q[X]_{\lambda}$
is the $\lambda$-summand as in \eqref{Xiso};
then the image of $\Q[\overline{G.v_I}]$
inside $\Q[X]$ contains at least those summands $\Q[X]_{\lambda}$
for $\lambda \in I$. 
 Using this, (iii) follows readily from 
an induction on weights and the facts (a) and (b).

Finally for the proof of (iv) we let $\mathcal{G}$ be the (Chevalley) integral model of $G$
over $\Z$; after inverting some primes we may suppose that
the chosen integral structure on $V_I$ is stable by $\mathcal{G}$. 
The image of the orbit map $\mathcal{G} \rightarrow \mathcal{X}$
has constructible complement and its codimension, intersected
with the generic fiber, is $\geq 2$ by point (ii). It follows from
the semicontinuity of fiber dimension 
that, with the exception of finitely many primes, the induced map
$\mathcal{G}/p \rightarrow \mathcal{X}/p$ also has image of codimension $\geq 2$.
Now, the stabilizer of the image of $v_I$ inside $\mathcal{X}/p$
is precisely $S_{\F_p}$ for almost all primes $p$. 
Consequently, the map $\mathcal{G}/p \rightarrow \mathcal{X}/p$ 
identifies $(S \backslash G)_{\F_p}$
with an open subset of the affine $\mathcal{X}_{\F_p}$
whose complement has codimension $2$. Therefore, $\mathcal{X}_{\F_p}$
is the affine closure of $(S \backslash G)_{\F_p}$, as claimed. 
\qed

The class of models constructed above has the following property:
It is constructed as a closed subscheme of a certain affine space over $\Z$,
and corresponding the linear functionals on this affine space 
 generate the ring of functions $\Z[\mathcal{X}]$. 
  This translates  to the following statement: the ring $\Z[\mathcal{X}]$ is generated by the image of

\begin{equation} \label{Imap} \bigoplus_{\lambda \in I}  \mathsf{V}_{\lambda}^{S} \otimes \mathsf{V}_{\lambda}^* 
\rightarrow \Z[\mathcal{X}],\end{equation} 
under the map of
\eqref{Xiso}. 
Here, as in (iii), $I$ is an arbitrary set of relevant weights that generate the lattice points in the relevant cone;
$\mathsf{V}_{\lambda}$ is the integral form of the $\lambda$-highest weight representation chosen in (iv), 
 and we take $\mathsf{V}_{\lambda}^*$ to be its $\Z$-linear dual inside the dual representation.

 Having defined the integral model $\mathcal{X}$, we will now prove a criterion for integrality:

\begin{prop} \label{intpoints}
Take $k=F_v$ a nonarchimedean local field. 

\begin{quote}
Let $g \in  B(F_v)$; then  $S g \in X(F_v)$
is integral, with respect to the integral model just defined, if and only if
for each relevant weight $\lambda$ 
\begin{equation} \label{Plambdaint}   u \in U \mapsto   \tilde{P}_{\lambda}(wug^{-1})\end{equation}
{\em a priori} in $F_v[U]$, lies in fact 
within $\mathcal{O}_v[U]$.   Here
  $w$ is any integral representative for the long Weyl element, and 
 $\tilde{P}_{\lambda} \in \Q[G]$ is as  in \eqref{Plambdadef},
uniquely characterized by $\tilde{P}_{\lambda}(bwu) =  b^{w \lambda} \bar{u}^{\langle \check{\rho}', \lambda \rangle}$.
\end{quote}

\end{prop}

 \proof   
 In the statement and in what follows, where we refer to $\mathcal{O}_v[U]$, we refer to the natural integral
model arising from taking the Zariski closure of $U$ inside the split form
$\mathcal{G}$ of $G$ over $\mathcal{O}_v$.

 By definition $x:=S g \in X(F_v)$ is integral, i.e. belongs to $\mathcal{X}(\mathcal{O}_v)$,  if and only if
 $f(x) \in \mathcal{O}_v$ whenever $f$ belongs
 to the ring of regular functions on the affine scheme $\mathcal{X}$.
Taking into account \eqref{Imap}, we see that, in order to show that
$Sg \in X(F_v)$ is integral,   it is sufficient to show that $f_{\lambda}(x) \in \mathcal{O}_v$
whenever $f_{\lambda}$ lies in the image of $\mathsf{V}^{S}_{\lambda} \otimes \mathsf{V}^*_{\lambda}$,
and $\lambda \in I$. 
By definition of the integral structure on $\mathsf{V}_{\lambda}^*$, this is  equivalent to asking that
\begin{equation} \label{inclusion} g^{-1} \mathsf{V}_{\lambda }^S \subset \mathsf{V}_{\lambda}\otimes \mathcal{O}_v.\end{equation}
We proved in \eqref{Plambdadef} that $V_{\lambda}^{S}$ corresponds, 
under \eqref{fdef}, to the $F_v$-linear span of the single function
$\tilde{P}_{\lambda}$. 

One readily sees,  then, that this function generates
the integral structure $\mathsf{V}_{\lambda}^{S}$.

The validity of \eqref{Plambdadef} is equivalent
to the right translate of this function through $g^{-1}$
remaining an integral section of the relevant line bundle
on the flag variety. Said  a little more explicitly:  {\em a priori} this right translate gives a  regular function on
$G_{F_v}$, which is given on $BwU$ by the rule
\begin{equation} \label{oinker}  b w u  \mapsto   \tilde{P}_{\lambda}(wug^{-1})\end{equation} 
 and we want it to extend to the split model of $G$ over $\mathcal{O}_v$. 

The condition stated in the proposition is then clearly
necessary for \eqref{inclusion}.
For sufficiency, assume that
\eqref{Plambdaint} holds; we will essentially apply ``Hartog's theorem.''
Let $\mathcal{G}$ be the smooth reductive group scheme of type $G$ over $\mathcal{O}_v$.  
\eqref{oinker} defines also a function (a section of the structure sheaf)  on $\mathcal{B} \times \mathcal{U}$,
where we use calligraphic font to denote the corresponding smooth group schemes over $\mathcal{O}_v$.
The morphism $(b,u) \mapsto bwu $ is an open immersion
$\mathcal{B} \times \mathcal{U} \rightarrow \mathcal{G}$
and thereby \eqref{oinker} defines a function
on its image, which we shall call the open cell in $\mathcal{G}$. 
The resulting section extends, moreover, to the union
of the open cell with the generic fiber of $\mathcal{G}$. 

In other words, we are given a scheme $\mathcal{G}$, an open set $V \subset \mathcal{G}$ whose codimension is $\geq 2$,
and a section of $\mathcal{O}$ over $V$. 
The algebraic form of Hartog's theorem asserts that the map
$\Gamma(\mathcal{G}, \mathcal{O}) \rightarrow \Gamma(V, \mathcal{O})$
is bijective.\footnote{ In detail, we choose a point $x$ in the complement of $V$ and let $R$ be the local ring
of the scheme $\mathcal{G}$ at $x$. 
The function on $V$ corresponds to an element of the fraction field of $R$
which belongs to every localization at a height one prime ideal, which therefore belongs to $R$,
i.e. extends over $x$.} This implies the sufficiency of the condition stated in the Proposition. 
\qed

 We can now prove the main result of this section, a classification of integral points: 
 
 \begin{prop} \label{intpointsX}
 Let $F_v$ be a nonarchimedean local field. Let   $ \chi \in X_*(A)$.

 Write $a=\langle \rho, \check{\rho} \rangle$.
 Then $$S u_t \varpi^{\chi} \in X(F_v)$$
 where $u_t$ is an arbitrary preimage of $t$ under $U \mapsto \mathbf{G}_a$,  is integral with respect
 to the integral model of
\S \ref{subsect: intModel}
  if and only if: 
\begin{itemize}
\item[(a)]
$\chi \in X_*(A)$ has nonegative pairing with each of
$ - \varpi_{\alpha} + n_{\check{\alpha}} \theta$ and
$  (a-1) \theta + \alpha-\rho$, where $\alpha$ varies over simple roots. 
 
\item[(b)]  The valuation of $t$ is bounded below by $  \langle \chi, \rho - (a-1) \theta \rangle$.
\end{itemize}
\end{prop}
 
For later use, we note that, as long as $a > 1$,
which excludes   $\SL_2$ and $\SL_2^2$,  we can summarize these conditions thus:
 \begin{equation} \label{kineq} \langle \chi, \theta \rangle  \mbox{is $\geq$ than both }    \langle \chi,  \frac{\rho-\alpha}{a-1} \rangle \mbox{ and } \langle \chi, n_{\check{\alpha}}^{-1} \varpi_{\alpha} \rangle \text{ for all simple }\alpha,\end{equation}
\begin{equation} \label{kineq2} v(t) \geq \langle \chi, \rho - (a-1) \theta \rangle. \end{equation}

\proof Applying the criterion of Proposition \ref{intpoints} to $g = u_t \varpi^{\chi}$, 
 we want to check if

$$u \mapsto \tilde{P}_{\lambda}(  w \varpi^{-\chi}  [\Ad(\varpi^{\chi})  u] u_{-t}  ) =  \varpi^{-\langle \lambda,   \chi \rangle} \tilde{P}(  \left[\Ad(\varpi^{\chi}) u  \right] u_{-t})$$
lies in $\mathcal{O}_v[U]$. 
 All that matters in the last expression is the projection of
 $\left[\Ad(\varpi^{\chi})  u \right] u_{-t}$ to $U^0 \backslash U \simeq \mathbf{G}_a$. 
Let $x_{\alpha}$, for $\alpha$ a simple root, be the various coordinates of $u \in U$
in the simple root spaces; then
this projection is given by 
$$ \sum\varpi^{\langle \chi, \alpha \rangle}  x_{\alpha}  - t.$$

We are therefore asking that the following expression, {\em a priori}
in $F_v[x_{\alpha}]$, actually lies inside $\mathcal{O}_v[x_{\alpha}]$:
\begin{equation} \label{BWbound} \varpi^{-\langle \lambda,    \chi \rangle} \left[
 \sum\varpi^{\langle \chi, \alpha \rangle}  x_{\alpha}-t  \right]^d\end{equation}
 where $d =\langle \lambda, \check{\rho}' \rangle$. 
The multinomial formula shows that the polynomial in $x_{\alpha},t$ that appears above is integral
if and only if the extremal coefficinets involving $x_{\alpha}^d$ and $t^d$ are integral.
 The integrality of these extreme coefficients corresponds to the integrality, for each $\alpha$,  of
$$\varpi^{-\langle \lambda,   \chi \rangle} \varpi^{d \langle \alpha, \chi \rangle}, \varpi^{-\langle \lambda,     \chi \rangle}  t^d$$
or formulated in terms of valuations, with the substitution $d =\langle \lambda, \check{\rho}' \rangle$. 
\begin{equation} \label{imp}   - \langle  \lambda,   \chi \rangle + \langle \lambda, \check{\rho}' \rangle  \langle \alpha, \chi \rangle \geq 0 \mbox{ and }
  - \langle  \lambda,   \chi \rangle  + \langle \lambda, \check{\rho}' \rangle v(t) \geq 0.\end{equation}
We want these linear inequalities to be valid for all characters $\lambda$ that lie in the relevant cone \eqref{relevant}.
Now, a linear inequality of the type \eqref{imp} is valid on all integral points in a homogeneous cone 
if and only if it is valid at all real points in the cone, for positive rescalings of integral points are dense;
 and, further, that is  the case if and only if it is
valid for each element in a set of generators for extremal rays. 
Consequently, it suffices to 
impose the conditions  \eqref{imp} above
solely at the set of extremal vectors prescribed in \eqref{basic0}.

For $\lambda = \varpi_\alpha - n_{\check{\alpha}} \theta$, so that
$\lambda$ is orthogonal to $\check{\rho}'$,  the inequalities coincide and become
$ \langle  - \varpi_{\alpha} + n_{\check{\alpha}} \theta,   \chi \rangle \geq 0.$
Moreover, these extreme weights impose no constraint on $t$. 
 The remaining weight is $\lambda = \rho - (a-1) \theta$. Recalling that $a = \langle \rho, \check{\rho} \rangle$,
which can be described equivalently as $a=\sum n_{\alpha} = \sum\check{n}_{\check{\alpha}}$, we see that the first inequality places the following constraint on $\chi$:
$ \langle -\lambda + \alpha, \chi \rangle \geq 0 $
 and the second inequality imposes a condition on $t$, namely,
 $$v(t) \geq \langle \rho-(a-1)\theta, \chi \rangle.$$
 
\qed

\section{Geometry of the nilpotent cone for $\mathfrak{pgl}_3$ and for $(2,2,2)$ tensors}
  \label{Examples}
 In the prior section, 
 we have studied the quotient $T^0U^0\backslash G$
 and its affine closure, where $G$ is a group with one-dimensional
 center, and $T^0$ is a one-dimensional torus. There are two examples of particular interest to us, one related to the adjoint $L$-function of $\mathrm{SL}_3$, and the other related to the triple product $L$-function of $\mathrm{SL}_2$.
 In these examples,  the space $X$ admits an explicit description
 in terms of the geometry of certain nilpotent cones, which we are going to recall.

 By studying these nilpotent cones with an explicit model, we observe that the integrality conditions obtained in Proposition \ref{intpointsX} could have been computed directly;
 actually, in our examples, the integral model is much simpler to describe in practice than in theory. 
 A special feature of these examples is that, 
at least after switching $G$ with its dual,
the space $X$ is a complete intersection in a certain irreducible $G$-representation with highest 
weight a central translate of $\rho$.

\subsection{Ginzburg's example, $G = \PGL_3 \times \Gm$. } \label{Ginzburgexample}

Here we take $G=\PGL_3 \times \Gm$; 
the isomorphism $\theta$ is the projection to $\Gm$, and $\check{\theta}$ the inclusion of the central $\Gm$ in the second factor;
we have in the notation of \S \ref{setup0}:
$$\check{\rho}'(t) = \bigg(\begin{bmatrix} t&0&0\\0&1&0\\0&0&t^{-1}\end{bmatrix}, t\bigg) \mbox{ and }  S = \{  \bigg(\begin{bmatrix}  t & x &y \\ 0 & 1 &-x \\ 0 & 0 & t^{-1} \end{bmatrix},t \bigg) \}$$

 Let $\mathfrak{pgl}_3= \mathfrak{sl}_3$ be its Lie algebra, 
 understood as trace-free $3 \times 3$ matrices, and take
\begin{equation} \label{gamma0def} x_0 = \begin{bmatrix} 0&1&0\\0&0&-1\\0&0&0\end{bmatrix} \in \mathfrak{pgl}_3.\end{equation}   Then the rule 
 $(g, z)  \mapsto z \cdot ( \mathrm{Ad}(g)^{-1} \cdot x_0)$
 identifies
 $S \backslash G$ to  an open subset of the nilpotent cone inside $\mathfrak{pgl}_3$, compatibly
 with right actions of $G$ where $\PGL_3$ acts by conjugation and $\Gm$ by scaling on the nilpotent cone.  
This induces an identification of the affine closure $X=\overline{S \backslash G}$ with that nilpotent cone:
 $$X = \mbox{nilpotent cone for $3 \times 3$ matrices} $$
 
 Let us now verify our prior results about integrality, namely Proposition \ref{intpointsX}, by hand (at least in sufficiently large characteristic, so that the integral model used in Proposition \ref{intpointsX}
 coincides with the integral model coming from the embedding into $3 \times 3$ matrices). 
 We use the same notation as in the Proposition.
 Letting $g$ be the product of the upper triangular matrix with $t$ in the $(1,2)$ entry and the cocharacter $(a,b,0;c)$ evaluated on $\varpi$, we calculate
 \begin{equation} \label{papa smurf} g=  \left(  \begin{bmatrix} 1&t&0\\0&1&0\\0&0&1\end{bmatrix}\begin{bmatrix} \varpi^a&0&0\\0&\varpi^b&0\\0&0&1\end{bmatrix} , \varpi^c \right)  \implies
 x_0 g= \begin{bmatrix} 0& \varpi^{c+b-a} & -t\varpi^{c-a}\\ & 0 & -\varpi^{c-b}\\&&0\end{bmatrix}\end{equation} 
 The conditions for integrality are thus $c \geq b$, $c \geq a-b$ and also $v(t) \geq a-c$. 
This is as follows from \eqref{kineq}. Indeed, here $k=c$, 
 the first inequality  of \eqref{kineq} shows that $c$ is larger than
   $\max(a-b,b)$; the  second condition  of \eqref{kineq} follows automatically, 
   for $n_{\alpha}^{-1} \varpi_{\alpha}$ lies in the convex hull of the $\frac{\rho-\alpha}{a-1}$. 
   The condition on $t$ is that $v(t)  \geq a  -c$.

\subsection{The dual to Ginzburg's example, $G=\SL_3 \times \Gm$}\label{Ginzburgexample2}

This is the dual of the previous case and we take again $\theta$ to be the projection to the $\Gm$ factor. 
The same construction as above gives
$$X = \overline{T^0 U^0 \backslash G} \twoheadrightarrow \mbox{nilcone for $\mathfrak{pgl}_3$}.$$
This map is generically 3 to 1, and $X$ is actually nonsingular in codimension $2$;
the only singularities occur at the preimage of the origin.  However, $X$ is not LCI at the origin.

\subsection{Smaller orbits in the Ginzburg examples} \label{smallorbitGinzburg}
The various duality proposals in \cite{BZSV} are quite stringent and entail geometric restrictions on lower-dimensional orbits. 
By restricting the definition of numerical weak duality to cusp forms we do not encounter most of this
rather interesting geometry. 
However, 
 it will be necessary for our later computations
to observe  all the smaller orbits have stabilizers that contain the unipotent radical of a proper parabolic subgroup.

Here, both in the case of \S \ref{Ginzburgexample} and \S \ref{Ginzburgexample2}, there are only two smaller orbits: the closed orbit (the origin), and the orbit of the rank one nilpotent matrix 
$$x_1 = \begin{bmatrix} 0&0&1\\0&0&0\\0&0&0\end{bmatrix}$$
for Ginzburg's example, and the preimage of the orbit of $x_1$ in the generically 3-1 cover for the dual to Ginzburg's example. 

In both cases, the stabilizer of $x_1$ (or a lift of $x_1$, in the dual to Ginzburg's example) is a 5-dimensional subgroup of the form $\Gm^2 \cdot U$, where $U$ is the upper triangular unipotent radical, and the $\Gm^2$ here is the kernel of the dual to $\theta$.

 \subsection{Garrett's example} \label{Garrettexample}

  Here we take
\begin{equation} \label{Garrettdef} G = \frac{\big\{(g_1,g_2,g_3,z) \in \mathrm{GL}_2^3 \times \Gm: z^3 = \mathrm{det}(g_1g_2g_3)\big\}}{\big\{(\lambda, \mu, \nu,1): \lambda \mu \nu = 1\big\}}\end{equation}
 and, again with the notation of \S \ref{setup0}, 
 $$\theta(g;z) = z, \check{\theta}(t) = \bigg(\begin{bmatrix} t^{1/2} & 0\\0&t^{1/2}\end{bmatrix}^{\times 3}, t\bigg) \mbox{ and } T^0 = \mathrm{Im}(\check{\rho}') = \bigg(\begin{bmatrix} t&0\\0&1\end{bmatrix}^{\times 3}, t\bigg).$$

We consider the $G$-representation on $\CC^2 \otimes \CC^2 \otimes \CC^2$ defined by
\begin{align*}
(v_1 \otimes v_2 \otimes v_3) (g_1,g_2,g_3,z) &= z^{-1}\big(v_1g_1 \otimes v_2g_2 \otimes v_3g_3\big)
\end{align*}
which we think of as a right $G \times \Ggr$-space acting on row vectors, with $\Ggr$ acting by   scaling.
Writing $x_0 $ for the symmetrization $e_{122} + e_{212} + e_{221}$ 
of $e_{122} := e_1 \otimes e_2 \otimes e_2$, the rule
$g \mapsto x_0 g$ identifies 
$X$ with the vanishing locus of the unique quartic semiinvariant
on this eight-dimensional space:
$$X = \mbox{ space of nilpotent $(2,2,2)$-tensors.} $$ 
The quartic semi-invariant is given by Cayley's hyperdeterminant.\footnote{
Thinking
of $(2,2,2)$ tensors  as a cube of numbers with side length 2, we take a pair of opposing faces which gives us two $2 \times 2$ matrices $A$ and $B$. Then we consider the binary quadratic form
$$Q(x,y) := \mathrm{det}(Ax+By)$$
whose discriminant $f = \mathrm{disc}(Q)$ is the quartic semiinvariant - this is the Cayley hyperdeterminant.}

Again  let us illustrate our prior results on integrality, with
the same caveats on characteristic as in 
\S \ref{Ginzburgexample}.
 In the notation of the previous section,  $n_{\alpha}=1/2$ and $a=3/2$; 
    for $\chi=(a_1, b_1, a_2, b_2, a_3, b_3; (a+b)/3)$,
the conditions of  \eqref{kineq} becomes\footnote{
The
second condition of \eqref{kineq}
is automatic,  for
  again $n_{\alpha}^{-1} \varpi_{\alpha}$ lies in the convex hull of the $\frac{\rho-\alpha}{a-1}$.}
 $$ \frac{k}{2}  = \frac{a+b}{6} \geq  (a-b)/2 - (a_i-b_i)   \iff  \frac{2b-a}{3}  \geq b_i -a_i  $$
and that $v(t)  \geq \sum (a_i-b_i)/2  -(a+b)/6$.
 Again this coincides with what we see by direct computation: let us use the notation
 $$u(t_1, t_2, t_3) := \left(\begin{bmatrix} 1 &t_1 \\0&1\end{bmatrix} \times \begin{bmatrix} 1 &t_2 \\0&1\end{bmatrix} \times \begin{bmatrix} 1 &t_3 \\0&1\end{bmatrix}, 1 \right)$$  
then the group element $(u(t,0,0)\varpi^{(a,b)}, \varpi^{\frac{a+b}{3}})$ sends our distinguished point $x_0$ to
\begin{equation} \label{exp} \varpi^{-\frac{a+b}{3}}\bigg(\varpi^{b_1+b_2+a_3}e_{221} + \varpi^{b_1+a_2+b_3}e_{212}+\varpi^{a_1+b_2+b_3}e_{122}+ t\varpi^{b_1+b_2+b_3}e_{222}\bigg)\end{equation}

\subsection{The dual to Garrett's example} \label{Garrettdual}
  On the dual side
\begin{equation} \label{Garrettdefdual} G = \frac{\big\{(g_1,g_2,g_3) \in \mathrm{GL}_2^3: \mathrm{det}(g_1) = \mathrm{deg}(g_2) = \mathrm{det}(g_3)\big\}}{\Delta\mu_3} \end{equation}
and here we take
 $$\theta(g) = \det(g_i)^3, \check{\theta}(t) = \begin{bmatrix} t^{1/6} & 0\\0&t^{1/6}\end{bmatrix}^{\times 3}, \check{\rho}'(t) = 
 \begin{bmatrix} t^{2/3} & 0\\0&t^{-1/3}\end{bmatrix}^{\times 3}.$$
 $$ S = \{ \bigg(\begin{bmatrix} t^2 & x\\ 0&t^{-1}\end{bmatrix}, \begin{bmatrix} t^2 & y\\ 0&t^{-1}\end{bmatrix}, \begin{bmatrix} t^2 & z\\ 0&t^{-1}\end{bmatrix} \bigg):x+y+z=0\}.$$

We claim that in this case:
$$ X = \mbox{ affine cone of the Pl\"ucker embedding of $\mathrm{LGr}(3,6)$}$$
where on the right we have the Grassmannian of Lagrangian 3-planes in $\CC^6$, into $\wedge^3 \CC^6$. In other words, every nonzero point of $\check{X}$ can be represented by a $(L,\nu)$, where $L \subset \CC^6$ is a Lagrangian and $\nu \in \wedge^3L$ is a volume form on the dual to $L$.   $G$
acts in the natural way. 
 Indeed, presenting $\C^6$ as the direct sum of $3$ symplectic planes $\langle e_i, f_i \rangle$ for $1 \leq i \leq 3$, and taking
$$x_0 =  (f_1-f_2) \wedge (f_2-f_3) \wedge (e_1+e_2+e_3).$$
 the rule $g \mapsto x_0 \cdot g$ defines an open immersion $T^0U^0 \backslash G \rightarrow X$
extending to an identification of the affine closure with $X$. 

\subsection{Smaller orbits in the Garrett example} \label{smallorbitGarrett}
For the same reasons as in  \S \ref{smallorbitGinzburg} we will analyze the 
smaller $G$-orbits. We will do this only over $\C$ (the results may remain
valid in all characteristics, but we didn't check).

We suggest that the reader skip this section upon a first reading, as only very crude features are needed for our computation. 
In what follows $B$ denotes the upper trianguar Borel subgroup of $\GL_2$
and for $b \in B$ the notation $b_{11}, b_{22}$ denote its upper and lower diagonal entry. 

In the case of \S \ref{Garrettexample}, in addition to the unique closed and open orbits,
there are three $5$-dimensional orbits and one $4$-dimensional orbit:
\begin{itemize}
\item Three five-dimensional orbits: we can take a tensor in 
$\C^2 \otimes \C^2 \otimes \C^2$
to have the form $v \otimes X$, where $v \in \C^2$ and $X \in \C^2 \otimes \C^2$ is generic, and there are three such families depending on which of the three tensor factors $v$ lies in.
For instance, we may take $v = e_1$ in the first tensor factor and $X = e_{12} - e_{21}$ in the second and third tensor factor as a representative point in its orbit, which gives us the point
$$v \otimes X =  e_{112} - e_{121}$$
and its stabilizer is the 6-dimensional subgroup 
$$ \{(b, g, g; z): b \in B, g\in \GL_2, z \in \Gm:
z = b_{11}  \det(g).\}.  $$

\item  A four-dimensional orbit: we can take a tensor of rank one. 
A representative point and its  stabilizer are:
$$e_{111}, \{(b_1, b_2, b_3, z):  \prod_j b_{j, 11} = z\}$$
where the $b_j$ belong to $B$.
In this case we also have $\prod b_{j, 11}^2 \cdot \prod b_{j, 22}^{-1}=1$;
the toral part of the  stabilizer is again the kernel to the dual of $\check{\rho}'$. 
\end{itemize}

We will  now describe the orbit structure of $G$ on $X$ in the situation of
\S \ref{Garrettdual}. In addition to the unique closed and open orbit,
there are three $5$-dimensional orbits, and one $4$-dimensional orbit:

 \begin{itemize}
\item Three five-dimensional orbits: Here the isotropic $3$-space
is the sum of a line in one of the three symplectic planes, and the graph of a
symplectic (anti-)isomorphism between the remaining two. 
A representative point  and stabilizer is given by
$$ e_1 \wedge (e_2-e_3) \wedge (f_2 - f_3),   \{ (b,g,g) \in B \times \Delta \GL_2 :   b_{22} \cdot \det(g) = 1\},$$
where $B$ is the upper triangular Borel. Note that $b_{11} b_{22} = \det(g)$
so we can rewrite the condition as $b_{11} b_{22}^2=1$.

\item  A four-dimensional orbit: the isotropic $3$-space is the sum of a line in each
of the three symplectic planes. A representative point and stabilizer is
\begin{equation} \label{4X}  
e_1 \wedge e_2 \wedge e_3,   \{ g = (b_1, b_2, b_3)  :    \prod b_{j,11}= 1\}. \end{equation}  
Again, the toral part of the stabilizer coincides with the kernel of the dual to $\check{\rho}'$. 
 \end{itemize}

\newcommand{\unlambda}{\underline{\lambda}}
\section{Unfolding to the Whittaker model} \label{sect: Unfolding}

We follow notation as in \S \ref{setup0};
in particular, we have a reductive group $G$ with one-dimensional center,  
and a conical $G$-variety $X = \overline{S \backslash G}$
with $S = T^0U^0$ defined by \eqref{eqn: cone}.
 
\begin{thm} \label{mainthm} Let $(G, \check{G})$ be as in \S \ref{Examples} --
that is to say, either 
\begin{itemize}
\item[(a)] $G=\PGL_3 \times \Gm$ or its Langlands dual $G = \SL_3 \times \Gm$, or
\item[(b)] $G \sim \SL_2^3 \times \Gm$ or its Langlands dual $G \sim \PGL_2^3 \times \Gm$ as in \eqref{Garrettdef};
\end{itemize}
We refer to either group in case (a) as the ``Ginzburg case'' and either group in case (b) as ``Garrett case.''  Consider the $G$-space and $\check{G}$-spaces
  $$X = \overline{T^0U^0 \backslash G} \mbox{ and } \check{X} = \overline{\check{T}^0\check{U}^0\backslash \check{G}}.$$
Set  $\gamma=1$ in the Ginzburg case and $\gamma=2$ in the Garrett case. 
Equip $X$ with the $\Ggr$ scaling action given by $\gamma \check{\theta}$.
 
 Then $X$ and $\check{X}$, regarded as $\Q$-varieties,  define  \wdd
 periods in the sense of Definition \ref{weakly dual def}; explicitly, the $X$-automorphic period matches with the $\check{X}$-spectral period and vice versa up to a power of $\Delta^{1/4}$:  
\begin{equation} \label{conclusion} P_X(f) =  \Delta^{3/4} L_{\check{X}}(\varphi_f), 
P_{\check{X}}(g) =  \Delta^{3/4} L_{X}(\varphi_g).
\end{equation}
In other words, the pair $(X, \check{X})$ has discrepancy $3$. 
 \end{thm}

  Lest the reader get lost in the equations, we say the key point in the proof:  {\em the  various inequalities defining the notion of integrality on $X$ match with the set of weights occurring in  functions on $\check{X}$,} and vice versa.  
The proof will be given in the rest of this section; in particular, we carry out the automorphic computation in \S \ref{theorem-autperiod},
and match it with the spectral computation in \S \ref{theorem-specperiod}.

 \subsection{Key features} \label{keyfeatures}

  To uniformize the treatment of the cases we use the notation 
\begin{equation} \label{gimeldef} a = \langle \rho, \check{\rho} \rangle, \mbox{ and }   \gamma = \frac{1}{a-1} \mbox{ and }
 \gimel
 = \frac{(2a-1) (2a-3)}{(2a-2)}\end{equation}

 In the Ginzburg case $a=2,\gamma=1, \gimel =3/2$ and
 the Garrett case $a=3/2, \gamma=2,\gimel=0$. 

Note that the prefactor of discriminant in \eqref{conclusion} is
$$\Delta^{3/4} = q^{2(g-1) \times 3/4} = q^{(3/2)(g-1)}$$   Note that $a>1$ in all cases;
 we use the positivity of $a-1$ implicitly.
 Moreover the following hold, in both cases:  
 \begin{itemize}
\item[(1)] 
  The cocharacter of $\check{G}$ dual to $\theta$ is given by $\gamma$ times $\check{\theta}_{\check{G}}$, and vice versa; that is to say, 
under the duality isomorphisms 
 $$X^*(T) \simeq X_*(\check{T}) \mbox{ and } X^*(\check{T}) \simeq X_*(T)$$
 we have
\begin{equation} \label{warning}  \theta_G \leftrightarrow \gamma \check{\theta}_{\check{G}}, 
\check{\theta}_G \leftrightarrow
  \gamma^{-1}  \theta_{\check{G}}. \end{equation}
   For example, consider the Garrett case, with $G$ the group defined by \eqref{Garrettdef} and $\check{G}$ the group defined by \eqref{Garrettdefdual}.  
  Recall that $\theta_G: G \rightarrow \Gm$ is the map 
  $$\theta_G(g_1,g_2,g_3;z) = \big(\mathrm{det}(g_1)\mathrm{det}(g_2) \mathrm{det}(g_3)\big)^{1/3} = z$$
  so it corresponds to the $\check{G}$-cocharacter 
$$t \mapsto \begin{bmatrix} t^{1/3} & 0\\ 0 & t^{1/3}\end{bmatrix}^{\times 3} = 2\check{\theta}_{\check{G}}(t)$$
Dually, $\theta_{\check{G}}: \check{G} \rightarrow \Gm$ is the character
$$\theta_{\check{G}}(g_1,g_2,g_3) = \mathrm{det}(g_1) \mathrm{det}(g_2) \mathrm{det}(g_3)$$
so it corresponds to the $G$-cocharacter 
$$t \mapsto \begin{bmatrix} t & 0\\0&t\end{bmatrix}^{\times 3} = (1/2) \check{\theta}_G(t)$$
  
  \item[(2)] 
 The points $n_{\check{\alpha}}^{-1}\varpi_\alpha$ lie in the convex hull of $\frac{\rho-\alpha}{a-1}$, as $\alpha$ ranges through simple roots. In particular, \eqref{automorphicbounds} reduces to the simpler form
 \begin{equation} \label{simplifiedBounds}\frac{\langle \chi, \rho\rangle}{a-1} \geq \langle \chi, \theta_G \rangle \geq \frac{\langle \chi, \rho-\alpha\rangle}{a-1}\end{equation}
 \end{itemize}

\subsection{Eigenforms and eigencharacters}

Note that at the identity coset in $X$, the top exterior power of the tangent space as a representation of $T^0$ is given by the weight
$$-(-2\rho + \alpha_0) \text{ for some (and any) simple root } \alpha_0$$
where the extra minus sign reflects the fact that we are taking the right adjoint action instead of the usual one. In particular, the weight with which $\Gm \stackrel{e^{\check{\rho}'}}{\rightarrow} T$ scales a volume eigenform $\eta$ on $X$ is
\begin{equation} \label{scaling factor} \langle \check{\rho}', 2\rho-\alpha_0\rangle = 2\langle \rho, \check{\rho}\rangle -1 = 2a-1.\end{equation}
and correspondingly the eigenform character of $G$ is given by a character extending this, i.e. $g \mapsto \theta(g)^{2a-1}$. 
Taking into account that the $\Ggr$ scaling action is defined to be $\gamma \check{\theta}$, 
the eigenform character  $\eta$ of \eqref{etadef} on $G \times \Ggr$ is given by
\begin{equation} \label{etaGGr} \eta(g, \lambda) = \theta(g)^{2a-1} \lambda^{(2a-1) \gamma}.\end{equation} 
In particular, in the notation of \eqref{LXdef} we have $\varepsilon = (2a-1)\gamma$.

\subsection{The automorphic period} \label{theorem-autperiod} 
To verify the conditions of Definition \ref{weakly dual def} we must compute
$P_X(f)$ (and dually $P_{\check{X}}(f')$)
whenever the automorphic coefficient field $\F$ has sufficiently large characteristic, 
and when $X_{\F}$ is taken to be the base change of a suitable integral model of $X$ over $\Q$.
We will in fact assume that $\F$ has large enough characteristic  so that the following two points hold:
\begin{itemize}
\item[(i)]
 (iv) of  \S \ref{subsect: intModel} is valid: that is to say,  the base-change of the integral model specified 
in \S \ref{subsect: intModel} is exactly the affine closure of 
 $ (S \backslash G)_{\F}$.  
 \item[(ii)] The conclusion noted in
\S \ref{smallorbitGinzburg} and \S \ref{smallorbitGarrett} in the case of characteristic zero remains valid for $\F$: 
 all the smaller  $G$-orbits  on the affine closure of $(S \backslash G)_{\F}$ have stabilizers that contain the unipotent radical of a proper parabolic subgroup:
 \end{itemize}
 In our examples, it is quite plausible that this does not exclude {\em any} characteristics, but we did not check this.

In what follows, then, we will use freely notation for automorphic forms as set up in \S \ref{notation}. 
Let $f$ be an automorphic form on $G$, which we will always assume to be unramified, generic, and cuspidal in the following.   
We normalize $f$ according to Whittaker normalization (see \eqref{Whitnormnew}).
Let $\Phi = \otimes\, \Phi_v \in \mathcal{S}(X_\mathbb{A})$ be the factorizable Schwartz function with local factors \begin{equation} \label{Phivdef}
\Phi_v(x) = \mathbf{1}_{X(\mathcal{O}_v)}(xe^{\gamma \check{\theta}}(\partial_v^{1/2})) \end{equation}
as dictated by the discussion around \eqref{Phidef}.

\begin{lem} \label{snt} Let $f$ be a Whittaker normalized cuspidal automorphic form on $G$, and let $\Phi = \otimes \Phi_v$ as defined above in \eqref{Phivdef}. Then the global automorphic period $P_X(f)$ unfolds into an Euler product
\begin{align*} 
P_X(f) &= q^{(g-1)\gimel} \prod_v \, \int_{A_v}\, (\eta \delta_B)^{1/2}(a) \left [\int_{F_v} \, \Phi_v(u_t a)\psi(t) \, dt\right] \, W_v^{\mathrm{ur}}(a a_{0,v}^{-1}) \, da\\
\end{align*}
where $u_t$ is a lift of $t \in F_v$ under $\Psi: U(F_v) \rightarrow F_v$, and the local integrals are normalized to give their maximal compact subgroups volume 1.   
\end{lem}

\begin{proof}
We will use the notation $[G]$, etc., as in \eqref{bracketnotation}. 
 By \eqref{PXdef}, the automorphic period attached to $X$ is given by
\begin{align*}
 P_X(f)   &=  \Delta^{-\dim(S)/4} |\partial^{1/2}|^{1/2}  \int_{[G]} \, f(g) \sum_{x \in X_F-\{0\} } \Phi(xg)\eta^{1/2}(g) \, dg\\
 \\    &\stackrel{(*)}{=} \Delta^{-\dim(S)/4} |\partial^{1/2}|^{1/2}  \int_{[G]} \, f(g) \sum_{x \in S_F \backslash G_F} \Phi(xg)\eta^{1/2}(g) \, dg\\
&= \left( \mbox{same} \right) 
\int_{S_F\backslash G_\mathbb{A}} \, \Phi(g)f(g) \eta^{1/2}(g) \, dg\\
&= 
\left( \mbox{same} \right)
\int_{S_\mathbb{A}\backslash G_\mathbb{A}} \Phi(g) \int_{[S]} \, f(sg) \eta^{1/2}(sg) \, ds \, dg
\end{align*}
where $ds$ here is the left Haar measure on $S$.  The equality labelled $(*)$ -- that is, the passage 
from $X_F-\{0\}$ to $S_F \backslash G_F$ -- arises as follows: The explicit computation of smaller orbits 
in  \S \ref{smallorbitGinzburg} and \S \ref{smallorbitGarrett}, together with the assumptions on the
characteristic of $\F$ mentioned at the start of this section,  shows that all orbits
except the open one contain a unipotent radical of a proper parabolic in the stabilizer group, and therefore  
the contribution of that orbit to $\theta$ is orthogonal to cusp forms.  On the other hand, the $F$-points
of the open orbit are indeed identified with $S_F \backslash G_F$ since the Galois cohomology of $S$
vanishes by the vanishing of $H^1(F, \Gm)$ and $H^1(F, \mathbf{G}_a)$. 

Writing $S = U^0 \times \Gm$ in the evident way,
and $\delta_S$ the modular character of $\Gm$ acting on $U^0$
 the inner integral becomes -- writing $\unlambda$ for $e^{\check{\rho}'}(\lambda) \in G$,
 so that  e.g. 

  $d(\mathrm{Ad}(\unlambda) u) = \delta_S(\unlambda) du$ -- 
\begin{align*}
& \int_{[S]} \, f(sg) \eta^{1/2}(sg) \, ds \\
&= \int_{[U^0][\Gm]} \, f(u\unlambda g)\eta(\unlambda g)^{1/2} \, \delta_S^{-1}(
\unlambda) \, d\lambda \, du^0\\
&\stackrel{\eqref{meas form}}{=} \Delta^{(u-1)/2} \, \int_{[U^0][\Gm]} \, f(u\unlambda g)\eta(\unlambda g)^{1/2} \, \delta_S^{-1}(\unlambda) \, d\lambda \, d^\psi u^0\\
&=  \Delta^{(u-1)/2} \, \sum_{\alpha \in F^\times} \, \int_{[U][\Gm]} \, f(u\unlambda g)\eta(\unlambda g)^{1/2} \delta_S^{-1}(\unlambda) \psi_\alpha(u) \, d\lambda \, d^\psi u
\end{align*}
where in the second to last step we switched to the Fourier self-dual measure in the $u$-variable writing $u=\dim(U)$ and at the final stage we  expand the $\delta$-funtion on $U^0 \backslash U$ in a Fourier series 
to replace an integral over $U^0$ by an integral over $U$. The identity
$$\psi(\mathrm{Ad}_{e^{\check{\rho}'}(\alpha)}(u)) = \psi_\alpha(u)$$
where $\psi_{\alpha}$ is obtained by postcomposing $U \rightarrow \mathbf{G}_a$
with multiplication by $\alpha$, allows us to combine the sum over $\alpha \in F^\times$ with the adelic integral over $\Gm$ to get 
\begin{align*}
&= \Delta^{(u-1)/2} \int_{[U]}\int_{\mathbb{A}^\times} \, f(u \unlambda g)\eta(\lambda g)^{1/2} \delta_S^{-1}(\unlambda) \psi(u) \, d\lambda \, d^\psi u\\
&=  \Delta^{(u-1)/2} \int_{\mathbb{A}^\times} \, W_f(\unlambda g) \eta(\lambda g)^{1/2} \delta_S^{-1}(\unlambda) \, d\lambda
\end{align*}

Returning to the original period $P_X$ and substituting the previous line into the inner integral over $[S]$, we have  
\begin{align*} \label{PXinter} 
P_X(f) &=  (\Delta^{(u-1)/2} \Delta^{-\dim(S)/4} |\partial^{1/2}|^{1/2}) \int_{S_{\mathbb{A}} \backslash G_{\mathbb{A}}} \Phi(g) \, \int_{\mathbb{A}^\times} \, W_f(\unlambda g) \eta^{1/2}(\lambda g) \delta_S^{-1}(\lambda) \, d\lambda \, dg\\
&=  (\Delta^{(u-1)/2} \Delta^{-\dim(S)/4} |\partial^{1/2}|^{1/2})\int_{U^0_{\mathbb{A}} \backslash G_{\mathbb{A}}} \, \Phi(g)W_f(g) \eta(g) \, dg\\
&= \int_{A_{\mathbb{A}}} \, (\eta \delta_B)^{1/2}(a) \, \left[ \int_{\mathbb{A}} \, \Phi(u_t a) \psi(t) \, dt\right] W_f(a) \, da\\
&=  Q \prod_v \, \int_{A_v}\, (\eta \delta_B)^{1/2}(a) \left [\int_{F_v} \, \Phi_v(u_t a)\psi(t) \, dt\right] \, W_v^{\mathrm{ur}}(a a_{0,v}^{-1}) \, da
\end{align*}

where in the second step we folded the inner $\mathbb{A}^\times \simeq U^0_{\mathbb{A}} \backslash S_{\mathbb{A}}$ integral back into the outer $S_{\mathbb{A}} \backslash G_{\mathbb{A}}$ integral, in the last step we used Whittaker normalization \eqref{shifted CS}, and $Q$ is the global constant
defined as 
$$Q =  \Delta^{(u-1)/2-\dim(S)/4} |\partial^{1/2}|^{1/2}   W^0_f=\Delta^{\gimel/2}.$$ 
Here we computed the exponent of $\Delta$ via 
$\dim(S) = u$, $|\partial^{1/2}|^{1/2} = q^{-\frac{(g-1)}{2}  \langle \eta, \gamma \check{\theta} \rangle}$ and \eqref{Whitnormnew} as
$$
  \frac{(u-1)}{2}  - \frac{u}{4}  - \frac{\langle \eta,  \gamma \check{\theta}\rangle}{4} +
\langle  \rho,  \rho^{\vee} \rangle - \frac{u}{4}  =  \langle \rho, \check{\rho}\rangle -\frac{1}{2}  
 - \frac{(2a-1)\gamma}{4}
= (a-1/2) (1-\gamma/2) =\frac{\gimel}{2}.$$
\end{proof}

Let $P_v$ be the local factor
$$P_v :=\int_{A_v} (\eta^{1/2}\delta_B^{-1})(a)\bigg[\int_{F_v} \Phi_v(u_t a)\psi(t) \, dt\bigg] \, W^{\mathrm{un}}_v(a a_{0,v}^{-1}) \, da$$
appearing in the statement of Lemma \ref{snt}. Then we compute
\begin{align*}
P_v &\stackrel{\eqref{Phivdef}}{=}  \int_{A_v} (\eta^{1/2}\delta_B^{-1})(a)\bigg[\int_{F_v} \mathbf{1}_{X(\mathcal{O}_v)}(u_t a\varpi^{\gamma m_v\check{\theta}})\psi(t) \, dt\bigg] \, W^{\mathrm{un}}_v(a a_{0,v}^{-1})  \, da \\
&=  \sum_{\chi \in X_*(A)} \, (\eta^{1/2}\delta_B^{-1})(\varpi^\chi) W^{\mathrm{un}}_v(\varpi^{\chi} a_{0,v}^{-1})  \bigg[\int_{F_v} \mathbf{1}_{X(\mathcal{O}_v)}(u_t \varpi^{\chi+\gamma m_v\check{\theta}})\psi(t) \, dt\bigg] 
\end{align*}

 Let us now compute the inner integral over $F_v$. According to the integrality inequalities \eqref{kineq} and \eqref{kineq2}, the following conditions on $u\varpi^{\chi+\gamma m_v\check{\theta}}$ guarantee the nonvanishing of the Schwartz function $\mathbf{1}_{X(\mathcal{O}_v)}$
 at the evaluation point in the integral:
 \begin{equation}\label{SchwartzFunctionBound}
 \langle \chi, \theta\rangle \geq   \mbox{ all of } \langle \chi,  \frac{\rho-\alpha}{a-1} \rangle  - \gamma m_v\mbox{ and } \langle \chi,  n_{\check{\alpha}}^{-1} \varpi_{\alpha} \rangle -\gamma m_v  \text{ for } \alpha \text{ simple roots}
 \end{equation}
\begin{equation}\label{volumeAdditiveIntegral}
v(t) \geq \langle \chi, \rho-(a-1)\theta\rangle - m_v
\end{equation}
Since the conductor of $\psi$ is $\varpi^{-2m_v}$, the $F_v$-integral vanishes  over the range defined by
\eqref{volumeAdditiveIntegral}, unless 
\begin{equation} \label{WhittakerFunctionBound} \langle \chi, \rho-(a-1)\theta\rangle - m_v \geq -2m_v \iff \frac{\langle \chi,\rho\rangle +m_v}{a-1} \geq \langle \chi, \theta\rangle 
\end{equation}
Note here we are using $a > 1$.

When the integral of $\psi(t) \, dt$ does not vanish, it contributes the volume of its range of integration, which is $q^{-\langle \chi, \rho-(a-1)\theta\rangle+m_v}$ according to \eqref{volumeAdditiveIntegral}, so we have
\begin{align*}
P_v &= \sum_{\chi}  (\eta^{1/2}\delta_B^{-1})(\varpi^\chi)W_v^{\mathrm{un}}(\varpi^\chi a_{0,v}^{-1}) \, q^{-\langle \chi, \rho-(a-1)\theta\rangle + m_v}\\
&\stackrel{\eqref{topolish}}{=} q^{-2m_v\langle \rho, \check{\rho}\rangle} \sum_{\chi} \eta^{1/2}(\varpi^\chi)s_{\chi+2m_v\check{\rho}}(\mathrm{Sat}_v)\, q^{(a-1)\big(\langle \chi, \theta\rangle + \gamma m_v\big)}
\end{align*}
where the sum is over $\chi$ that satisfy both
\eqref{WhittakerFunctionBound} and 
\eqref{SchwartzFunctionBound}; and
$s_\chi$ is the Weyl character with highest weight $\chi$ if $\chi$ is dominant, and 0 otherwise. 
Also, at the last stage, we grouped several powers of $q$ as follows to obtain the final prefactor: $\delta_B^{-1}(\varpi^{\chi}) \cdot
q^{-\langle \chi+2 m_v \check{\rho}  , \rho  \rangle } \cdot q^{-\langle \chi, \rho \rangle} = q^{-2 m_v \langle \check{\rho}, \rho \rangle}$, where the middle factor $q^{-\langle \chi+2m_v \check{\rho}, \rho\rangle}$ comes from the Casselman--Shalika formula \eqref{Wvun}.

 \subsubsection{Reindexing} \label{reindexing} 
 This subsection \S \ref{reindexing} carries
 out some tedious reindexing which involves only elementary algebra; the reader
 might wish to skip directly to the conclusion
 Proposition \ref{thm: PX}.
  Let us make the change of variables
$$\chi_{\textrm{new}} =   \chi_{\textrm{old}}+ 2m\check{\rho}$$
 then the above expression becomes 
\begin{align*}
P_v &= q^{-2m_v\langle \rho, \check{\rho}\rangle} \, q^{m_v\langle \eta, \check{\rho} \rangle} \sum_{\chi} \eta^{1/2}(\varpi^\chi)s_\chi(\mathrm{Sat}_v)\, q^{(a-1)\big(\langle \chi, \theta\rangle + \gamma m_v\big)}\end{align*}
where now $\chi = \chi_{\mathrm{new}}$ is now ranging over dominant coweights of $A$ subject to the following bounds: \eqref{WhittakerFunctionBound} gives rise to
$$\frac{\langle \chi_{\mathrm{new}}-2m\check{\rho}, \rho\rangle + m_v}{a-1} \geq \langle \chi_{\mathrm{new}}, \theta\rangle \iff \frac{\langle \chi_{\mathrm{new}}, \rho\rangle }{a-1} - (2+\gamma)m_v \geq \langle \chi, \theta\rangle$$
and \eqref{SchwartzFunctionBound} gives rise to 
\begin{align*}
&\langle \chi_{\mathrm{new}}, \theta \rangle \geq \text{ all of } \langle \chi_{\mathrm{new}}-2m_v\check{\rho}, \frac{\rho-\alpha}{a-1}\rangle -\gamma m_v \text{ and } \langle \chi_{\mathrm{new}}-2m_v\check{\rho}, n_{\check{\alpha}}^{-1}\varpi_\alpha\rangle - \gamma m_v \\
&\iff \langle \chi_{\mathrm{new}}, \theta \rangle \geq \frac{\langle \chi_{\mathrm{new}}, \rho-\alpha\rangle}{a-1} - (2+\gamma)m_v \, \text{ and } \langle \chi_{\mathrm{new}}, n_{\check{\alpha}}^{-1}\varpi_\alpha\rangle - (2+\gamma)m_v 
\end{align*}
Combining these into a string of inequalities and dropping the subscript ``new", we arrive at
\begin{equation}
\frac{\langle \chi, \rho\rangle}{a-1} \geq \langle \chi, \theta\rangle + (2+\gamma)m_v \geq \frac{\langle \chi, \rho-\alpha \rangle}{a-1} \, \text{ and } \,  \langle \chi, n_{\check{\alpha}}^{-1}\varpi_\alpha\rangle 
\end{equation}

We make another substitution on $\chi$, this time only affecting the center  
$$\chi_{\textrm{new}} = \chi_{\textrm{old}} + (2+\gamma)m_v \check{\theta}$$
Then $\chi_{\textrm{new}}$ ranges over dominant coweights satisfying the following bounds: 

\begin{equation}\label{automorphicbounds}  
\frac{\langle \chi, \rho\rangle}{a-1} \geq \langle \chi, \theta\rangle \geq \frac{\langle \chi, \rho-\alpha \rangle}{a-1} \, \text{ and } \, \langle \chi, n_{\check{\alpha}}^{-1}\varpi_\alpha\rangle
\end{equation}
and the local period becomes 
\begin{align*}
P_v &= q^{ -2m_v\langle \rho, \check{\rho}\rangle} \, q^{m_v\langle \eta, \check{\rho} \rangle} \sum_{\chi: \eqref{automorphicbounds}} \eta^{1/2}(\varpi^{\chi-(2+\gamma)m_v\check{\theta}})s_{\chi-(2+\gamma)m_v\check{\theta}}(\mathrm{Sat}_v)\, q^{(a-1)\big(\langle \chi, \theta\rangle - 2m_v\big)}\\
&= q^{-2m_va}  q^{-2(a-1) m_v} q^{\big(\frac{(2+\gamma)\langle \eta, \check{\theta}\rangle }{2}\big)m_v} (\chi_f \circ \check{\theta})(\partial_v^{-1-\gamma/2})  \sum_{\chi: \eqref{automorphicbounds}} q^{\langle \chi, (a-1)\theta-\eta/2\rangle} \, s_\chi(\mathrm{Sat}_v)
\end{align*}
where we use the central character $\chi_f$ of the automorphic form $f$ to express
$$s_{\chi-(2+\gamma)m_v\check{\theta}}(\mathrm{Sat}_v) = (\chi_f \circ \check{\theta})(\partial_v^{-1-\gamma/2})s_\chi(\mathrm{Sat}_v)$$

Using the relation $\eta = (2a-1) \theta$
 from \eqref{etaGGr} for the eigenform character on $G$, 
we can express the power of $q$ outside the summation more succinctly in terms of the constant $\gimel$ we defined 
in \eqref{gimeldef}: $$-2ma -2(a-1)m + (2+\gamma)(2a-1)m/2 = (2a-1) m (-1+\gamma/2) = -m\gimel,$$
We arrive at the following final expression for our local period:
$$P_v = q^{-m_v \gimel} (\chi_f \circ \check{\theta})(\partial_v^{-1-\gamma /2})  \sum_{\chi: \eqref{automorphicbounds}} q^{\langle \chi,  -\theta/2 \rangle} \, s_\chi(\mathrm{Sat}_v).$$

Returning to the formula given by Lemma \ref{snt} for the global period 
$P_X(f) = \Delta^{\gimel/2} \prod_v P_v$
and combining all the local constants $\prod_v \, q^{-m_v\gimel} = \Delta^{-\gimel/2}$, we arrive at the main result on the automorphic side, which we summarize: 
 
\begin{prop}\label{thm: PX} Let $(G, \check{G})$ be as in the statement of Theorem \ref{mainthm}. Let $X = \overline{T^0U^0\backslash G}$ be viewed as a conical $G$-variety with $\Ggr$-acting through $\gamma \check{\theta}_G$. Suppose that the automorphic coefficient field $\mathbf{F}$ has sufficiently large characteristic to apply  (iv) of  \S \ref{subsect: intModel}.
 Then the automorphic $X$-period of an unramified cusp form $f$ on $G$ is given by an Euler product
$$P_X(f) =  (\chi_f\circ \check{\theta})(\partial^{-1-\gamma/2})) \, \prod_v \, \sum_{\chi: \eqref{automorphicbounds}} q^{-\langle \chi, \theta/2\rangle} \, s_\chi(\mathrm{Sat}_v)$$
where $\chi_f$ is the central character of $f$, $\gamma = \frac{1}{\langle \rho, \check{\rho}\rangle -1}$, $\mathrm{Sat}_v$ is the Satake parameter of $f$ at a place $v$, and $\chi$ ranges through dominant coweights satisfying \eqref{automorphicbounds}. 
\end{prop}

\subsection{Comparison with spectral period for $\check{X}$} \label{theorem-specperiod} 

Having computed the automorphic $X = \overline{T^0U^0\backslash G}$-period, we compare it with the spectral period of 
$\check{X} := \overline{\check{T}^0 \check{U}^0\backslash \check{G}}$.
The content of this comparison amounts to 
matching the set of weights in Proposition \ref{thm: PX} --
i.e. the various $\chi$ satisfying \eqref{automorphicbounds} --  
with the set of weights occurring in the function ring of $\check{X}$.

By definition (\S \ref{conical} and \eqref{LXdef0prod} and \eqref{LXdef0}, also using
\S \ref{etaGGr} to compute $\varepsilon$) the normalized spectral period of a $\check{G}$-valued $L$-parameter $\varphi$ is  \begin{equation} \label{LXspectral} L_{\check{X}}(\varphi) =
\Delta^{\frac{(2a-1)\gamma- \dim(G/U)}{4}} \mathfrak{z}\prod_v \, \mathrm{tr}\big(q^{-1/2} \times \mathrm{Fr}_v \, | \,
\kk[\check{X}]\big)\end{equation} 
The power of $\Delta$ equals $-3/4$:
 \begin{equation} \label{eval-exp}  \mathrm{dim}(G/U) - (2a-1)\gamma =  3.\end{equation} 

Recall that $\mathfrak{z}$ is defined in \eqref{zdef},
and here it can be computed 

using \eqref{warning}, \eqref{etaGGr}
in terms of the central character $\chi_f$ as 
$$(\chi_f \circ  \check{\theta}_G^{(2a-1)\gamma})(\partial^{-1/2}) = (\chi_f \circ \check{\theta}_G)(\partial^{-1-\gamma/2})$$
which agrees exactly with the central character factor appearing in Proposition \ref{thm: PX}.

For the rest of this section, for notational simplicity we will use all the notation ``for $\check{G}$" instead of ``for $G$" as we have done in the automorphic calculation; more precisely, we follow these conventions: 

\begin{center}
$\check{\rho}, \check{\theta}_{\check{G}}, \check{\rho}'$ are all cocharacters into $\check{G}$, and $\chi$'s will be highest weights for $\check{G}$.   
\end{center}
According to Proposition \ref{Xfunction}, we can calculate the graded $\check{G}$-representation on $k[\check{X}]$ as
the direct sum of $V_{\chi}^* = V_{-w\chi}$ over ``relevant'' $\chi$ where $w$ is the long Weyl element. Writing $\check{\rho}' = \check{\rho} + \check{\theta}$, the relevancy criterion given by equation \eqref{relevant} reads as
$$0 \leq \langle \chi, \check{\rho}+ \check{\theta}_{\check{G}} \rangle \leq \underset{\check{\alpha} \text{ simple}}{\mathrm{min}} \langle \chi, \check{\alpha}\rangle$$
Subtracting $\langle \chi, \check{\rho}\rangle$ from the previous line and taking the negative of these inequalities, we arrive at
\begin{equation} \label{spectralbounds} \langle \chi, \check{\rho} \rangle \geq \langle \chi, -\check{\theta}_{\check{G}} \rangle \geq \underset{\alpha \text{ simple}}{\mathrm{max}}\,\langle \check{\rho} -\check{\alpha}, \chi\rangle\end{equation}
Now, replacing $\chi$ by  $\chi' := -w\chi$, we get $\kk[\check{X}] \simeq \bigoplus V_{\chi'}$ where
$\chi'$ satisfies 
\begin{equation} \label{spectralbounds2} \langle \chi', \check{\rho} \rangle \geq \langle \chi', \check{\theta}_{\check{G}} \rangle \geq \underset{\alpha \text{ simple}}{\mathrm{max}}\,\langle \check{\rho} -\check{\alpha}, \chi'\rangle\end{equation}
By \eqref{warning}, under the isomorphism $X_*(\check{A}) \simeq X^*(A)$, 
$\check{\theta}_{\check{G}}$ corresponds to $\gamma^{-1} \theta_G=(a-1)\theta_G$; so we see that  \eqref{automorphicbounds} and \eqref{spectralbounds2} are exactly equivalent --
recalling that the second set of inequalities in \eqref{automorphicbounds} follows from the first by the convex hull property.

Moreover the $\Ggr$ grading 
on $V_{\chi}$
is given by $\langle \chi, \gamma \check{\theta}_{\check{G}}\rangle\stackrel{\eqref{warning}}{=}
\langle \chi, \theta_G \rangle$.  Consequently, the local factor in the product of \eqref{LXspectral} is given by
$$
\mathrm{tr}\big(q^{-1/2} \times \mathrm{Fr}_v \, | \,
\kk[\check{X}]\big) = \sum_{\chi : \eqref{spectralbounds2}}  q^{-\langle \chi, \theta_G \rangle/2}  s_{\chi}(\mathrm{Sat}_v),
$$
where we sum over dominant weights $\chi$ satisfying \eqref{spectralbounds2}. Now comparing with  Proposition \ref{thm: PX} we can rewrite  
\eqref{LXspectral}
as 
$$ L_{\check{X}}(\varphi) = \Delta^{-3/4} P_X(\varphi),$$
i.e. coinciding with $P_X(\varphi)$ from Theorem \ref{thm: PX} but for the factor of $\Delta^{-3/4}$. 
This concludes the proof of Theorem \ref{mainthm}. 

\begin{rem} In the automorphic period computation we produced a set of inequalities \eqref{automorphicbounds} on the dominant cocharacters of $G$, which, in the Ginzburg and Garrett examples reduces to simpler inequalities as in \eqref{simplifiedBounds} since one set of inequalities lies in the convex hull of the other. We believe that this phenomenon only occurs in cases isogenous to the Ginzburg and Garrett examples, but we leave the general automorphic calculation available since it will be essential for a deeper understanding of weak duality for these singular spaces. 
\end{rem}

\appendix

\section{Comparison with \cite{BZSV}}\label{BZSVcomp}

The following remarks are not necessary to read this paper,
but are designed to help the unfortunate reader  or author
who needs to interface the conventions here with those appearing in \cite{BZSV}.

 \subsection{Left and right actions} 

\cite{BZSV} follows the following left right convention:
\begin{itemize}
\item 
when we consider $(G, X)$ automorphically,  we take $G$ to act on the right;
\item 
when we consider $(G,X)$ spectrally, we take $G$ to act on the left; and
\item we pass 
one to the other in the usual way (i.e. via inversion on $G$)
but on $\GGm$ we do nothing. 
\end{itemize}

Thus, for example, taking $G=\SL_2, X= \mathbf{A}^2$, and $\GGm$ to act by scaling means:

\begin{itemize}
\item considered automorphically, the action is $x \cdot (g, \lambda) = xg\lambda$;
\item considered spectrally, the action is $(g, \lambda) \cdot  x = x g^{-1} \lambda$.
\end{itemize}

Taking $X=G/H$ and $\GGm$ acting trivially means:
\begin{itemize}
\item  considered automoprhically, we are taking $X=H \backslash G$ as a right $G$-space;
\item considered spectrally, we take $X=G/H$ as a left $G$-space.
\end{itemize}

In the BZSV convention, the Tate period is $G=\check{G}=\Gm$
and $X=\check{X} = \mathbf{A}^1$ and all actions are by scaling;
but the actions on $\check{X}$ are left actions. 

 In this paper we have worked exclusively 
 with right $G$ and $\check{G}$-spaces;
when we say that $(G, X)$ and $(\check{G}, \check{X})$
are dual, it corresponds to  the same duality in the BZSV notation where
the left action of $\check{G}$ on $\check{X}$ is defined via:
\begin{equation} \label{BZSVswitch} \check{g} \cdot_{\mathrm{BZSV}} \check{x} = \check{x} (\check{g}^d)^{-1} \end{equation}
with $d$ the duality (=Chevalley, more or less) involution. 
Thus, 
in our normalization,  the Tate period is $G=\check{G}=\Gm$
and $X=\check{X} = \mathbf{A}^1$ and all actions are by scaling
and now all actions remain right actions. 

\subsubsection{Which $L$-function appears: comparison with \cite{BZSV}}  \label{whichL}  
Let us consider now the situation when $\check{X}$ is a vector space. (We are only concerned
here with straightening out some issues of left versus right, so the global geometry of $\check{X}$ is irrelevant.) 
Ignoring shifts and normalization factors, the conjectures in Chapter 14 of \cite{BZSV} have the form
$$ \langle P_X, f \rangle \sim L(\check{X}, f ).$$
 Then the above $L$-function is the usual one whose local factor the characteristic polynomial of Frobenius on the vector space $\check{X}$.

 Now, our version of the same duality replaces the left action of $\check{G}$ on $\check{X}$
 by the right action given by $\check{x} \check{g} = (\check{g})^d \check{x}$. 
In this right action,  the associated left action of $\check{G}$ on $k[\check{X}]$
 precisely coincides with the  (left) action of $\check{G}$ of the symmetric algebra of the (left) action of $\check{G}$ on $\check{X}$;
 this reduces to the fact that the involution $d$ inverts a maximal torus and thereby, by considering characters,
 replaces $\check{X}$ with its dual.
  
\subsubsection{Switching automorphic and spectral sides: comparison with \cite{BZSV}} 
In particular, suppose that $(G, X)$ and $(\check{G}, \check{X})$ are dual in the notation of this paper. Both sides have right action.
Then converting the spectral side to a left action via \eqref{BZSVswitch}, we arrive at a dual pair in the sense of \cite{BZSV}:
$$(G, X)\stackrel{\mathrm{BZSV}}{\longleftrightarrow}\left(\check{G}, \check{X}, \check{g} \cdot \check{x} = \check{x}(\check{g}^d)^{-1} \right)$$
Switching automorphic and spectral sides according to the ``BZSV convention'' by inversion gives 
$$ (\check{G}, \check{X},  \check{x}  \cdot \check{g} = \check{x}\check{g}^d )
\stackrel{\mathrm{BZSV}}{\longleftrightarrow} (G, X, g \cdot x = x g^{-1}),$$
and finally switching back to the notation of our paper, i.e., undoing \eqref{BZSVswitch} to obtain a right action on the spectral side, we have
$$ (\check{G}, \check{X},  \check{x}  \cdot \check{g} = \check{x} \check{g}^d )
\stackrel{\mathrm{BZSV}}{\longleftrightarrow} (G, X, x \cdot g = x g^d),$$
In other words ``transferring between our notation and \cite{BZSV}''
does not commute with ``switching $G$ and $\check{G}$,''
there is a twist through the dualizing involution.

\subsection{Nonlinear $L$-functions} \label{QuillenSection}  
We will outline why the definition of the $L$-function attached to $\check{X}$ 
(given e.g. in \eqref{eq: nonlinearLdef}) is compatible with the one studied in \cite{BZSV}. 
We do not however attempt to give a treatment with full proofs, which would be very tiring both for us and for you; for example, we will 
ignore issues such as convergence, the topology on the Galois group, and the technicalities of derived rings.

For this section let us 
write $\Gamma$ for the {\'e}tale fundamental group (=unramified Galois group) of the curve
$\Sigma$ and $\Gamma^0 \subset \Gamma$ for the geometric {\'e}tale fundamental group
(=unramified Galois group of the curve over $\overline{\F_q}$).   

To simplify our life we suppose that (i)  $\check{X}$ is conical (ii) that  the genus of $\Sigma$  is $2$ or greater
and (iii)  that $\Gamma^0$ fixes only the origin of $\check{X}$.

According to \cite{BZSV} the appropriate $L$-function
should be the Frobenius trace on the $L$-sheaf defined therein,
which, in words, amounts to 
\begin{quote} (*) the trace of  $t \times \mathrm{Frobenius}$  on the space of  functions
on the derived fixed points of $\Gamma^0$ on $\check{X}$.
\end{quote}
Here    $t \in \Gm(k)$ is a parameter playing the role of $q^{-s}$
in the usual definition of $L$-functions;
actually, \cite{BZSV} involves volume forms in place of functions, but our conventions
introduce an extra dualization, cf. \S \ref{whichL},  so we get functions instead.\footnote{Also, 
\cite{BZSV} specializes the parameter $t$ to $q^{-1/2}$.}
 
\subsubsection{A result of Quillen} 
We need to recall a result of Quillen \cite{Quillen}. Suppose that $R = \bigoplus R_n$ is a graded $k$-algebra
with $R_0 = k$ and all $R_n$ finite-dimensional over $k$. 
Quillen proves  (see Theorem 11.1 of \textit{loc. cit}) that the
formal sum $\sum (\dim R_n) t^n$ -- i.e. ``the trace of $t \in \Gm$
acting on functions on $\spec R$'' -- can be expressed as an infinite product 
\begin{equation} \label{Q1}  
 \sum (\dim R_n) t^n=\prod_{m \geq 1, q \geq 0} (1-t^m)^{(-1)^{q+1} t_{q,m}}\end{equation} 
where $t_{q,m}$ is the dimension of the $m$th graded piece
of the $q$th cohomology $T_q$ of the tangent complex of $R/k$. 
For each $m$ the product over $q$ is finite.\footnote{
Note that Quillen's formula has a sign $(-1)^q$ and replaces $t_{q,m}$
by the   dimension of the $m$th graded piece of $D_q(k/R)$. However this $D_q(k/R)$ is the shift by $1$
of the tangent complex for $R/k$ by an application of
the long exact sequence \cite[Theorem 5.1]{Quillen} for $k \rightarrow R \rightarrow k$; thus our sign $(-1)^{q+1}$.
 }
The proof of this identity proceeds
by extracting dimensions in Quillen's spectral sequence
relating, on the one hand, the symmetric power on the (derived) tangent space,
and, on the other hand, $\mathrm{Tor}^R(k,k)$. 
Now
suppose $\phi$ is a ring automorphism of $R$, preserving the grading; 
by the same reasoning we find that
\begin{equation} \label{Q2} \mbox{``trace of $t\phi$ on $R$'' }:=  \sum_{n} \mathrm{tr}(\phi|R_n) t^n = \prod_q \left( \mbox{char. poly. of $t\phi$ on $T_{q}$ }\right)^{\pm 1},\end{equation} 
where, formally, we rewrite the right hand side as a product over $m,q$ and interpret
as in \eqref{Q1}, and the signs are the same as in \eqref{Q1}; indeed \eqref{Q1}
is the specialization of \eqref{Q2} to $\phi=1$. 

\subsubsection{Back to our world} 
We will use \eqref{Q2} twice  to evaluate (*) and show that it coincides
with the formula \eqref{LXdef0prod} given in  \S \ref{conical}.

Having fixed a homomorphism $\Gamma \rightarrow \check{G}$,  $\check{X}$ acquires an action of 
$\Gamma$.
 Let us consider the 
 the derived fixed points of $\Gamma^0$ on $\check{X}$. These should be regarded here as the
spectrum of a certain derived ring  $R$ (formally: a simplicial $\kk$-algebra);
the classical ring associated to this derived ring is a nilpotent $\kk$-algebra
because we suppose that $\Gamma^0$ fixes only the origin in $\check{X}$. 

Now a fundamental design feature of derived invariants
is that they ``commute with taking tangent spaces,''
that is to say, the tangent space to derived invariants
should be the derived invariants on the tangent space -- 
and the derived invariants on a vector space are simply group cohomology.
Correspondingly, 
{\em the tangent complex for the $\Gamma^0$-derived  fixed points on $\check{X}$, at the origin, 
should be obtained by taking  geometric {\'e}tale cohomology upon the tangent complex for $\check{X}$.}
(We say ``should be'' because to give this sentence a precise meaning requires one to give more details
about how one treats the topology on $\Gamma^0$, etc. -- matters we are ignoring).

The definition of tangent complex goes over to simplicial $\kk$-algebras without change.
We apply \eqref{Q2} to $R$ as above, taking $\varphi$ to be the Frobenius automorphism. 
We see that 
the trace described in (*) equals 
\begin{equation} \label{Quilleneq}  \prod_{i, q} \det(1- tF|H^i(T_q))^{\pm 1}\end{equation} 
where the sign $\pm$ is given by the parity of $i+q+1$,   $H^i$ means
geometric {\'e}tale cohomology, and $T_q$ is the $q$th cohomology
for the tangent complex of $\check{X}$ at $0$. 

  Each product over $i$ can be evaluated by the
 Grothendieck--Lefschetz trace formula as itself a product
 over places $x$ of the curve; and thus we get 
$$ \prod_{x} \left[ \prod_{q} \det(1-t \mathrm{F}_x| T_q)^{\pm 1} \right ],$$
where the parity is now $q+1$, and $\mathrm{F}_x$ denotes the Frobenius for the
closed point $x$. 
By a second application of  \eqref{Q2}, 
each internal product
is none other than the trace of $\mathrm{F}_x$ 
on $\kk[\check{X}]$.  This argument shows, then, that (*) can be expressed by the  product over $x$ of 
the trace of   local Frobenius (times $t$) on the function ring of $\check{X}$. 
That is the definition of nonlinear $L$-functions that we took as our starting point
in 
\eqref{eq: nonlinearLdef} and \eqref{LXdef0}.

 \subsection{Normalizing factors for $L$-functions: 
comparison of \eqref{LXdef} with BZSV}  
We explain how  numerical weak duality is compatible with the conjectures of \cite{BZSV} in the case when $\check{X}$ is smooth.
For this we must first recall some terminology from there.    

 The choice of square root of the canonical bundle on $\Sigma$ permits us
    to form a ``normalized'' $L$-function  
       attached to a representation $\rho: \mathrm{Gal}_F \rightarrow \GL(T)$. 
       First of all, we define as in \cite[\S 11.2]{BZSV} a square root of the $\epsilon$-factor, by 
       $$  \epsilon(s,T)^{1/2} := \det(T)(\partial^{1/2})q^{(1/2-s)(g-1)\mathrm{dim}(T)},$$
       and then define, again following \cite[\S 11.2]{BZSV},  the normalized $L$-function via
\begin{equation} \label{Lnormdef}
L^{\mathrm{norm}}(s,T) := \epsilon(s,T)^{-1/2}L(s,T), \ \ 
\end{equation}
          This is symmetric under  under $(T,s) \leftrightarrow (T^\vee, 1-s)$.
          Indeed, the usual relation
    $L(s,T) = \epsilon(s,T)L(1-s,T^\vee)$ implies that
    $ L^{\mathrm{norm}}(T,s) = L^{\mathrm{norm}}(T^\vee, 1-s)$.
Next, for $T$ a $\check{G} \times \mathbf{G}_{\mathrm{gr}}$-representation and $\varphi: \gal_F \to \check{G}$ an $L$-parameter, we view $T = \bigoplus_k \, T_k$ as a graded $\gal_F$-representation, with grading coming from $\mathbf{G}_{\mathrm{gr}}$. Then we write 
(see \cite[\S 2.8]{BZSV})
    $$L(s,T^\fltns) := \prod_k \, L(s+k/2,T_k), L^{\mathrm{norm}}(s,T^{\fltns}) = \prod_k  L^{\norm}(s+k/2, T_k)$$
    for the $L$-function ``sheared" by the $\mathbf{G}_{\mathrm{gr}}$-grading.
The numerical conjecture
of \cite[\S 14.2]{BZSV} implies that, for a Whittaker normalized form $f$ with $L$-parameter $\varphi: \gal_F \to \check{G}$, we have
    \begin{equation} \label{FromBZSV} P_X^{\norm}(f)= \sum_{x \in \mathrm{Fix}(\varphi, \check{X})} \, L^{\mathrm{norm}}(0,T_x^\fltns),\end{equation} 
    assuming, as we will in what follows, that the fixed point locus of $\varphi$ on $\check{X}$ is discrete.
    (Note that \cite{BZSV} has a factor $q^{-b_G/2}$, but in the normalization of \cite{BZSV}
    the Whittaker period is also equal to $q^{-b_G/2}$, and so for the Whittaker normalized form this factor cancels out.) 
     
Observe from \eqref{Lnormdef}
that the contribution of the $\epsilon$-factor to $L^{\mathrm{norm}}(0, T_x^{\fltns})$
 equals
 $$ \det(T_x)(\partial^{-1/2}) \times \prod_{k}  q^{ (k/2-1/2) (g-1) \dim T_k}.$$
 But the sum $\sum k \dim(T_k)$ computes $\varepsilon$ from \eqref{etadef}, the weight of $\Ggr$ on the eigenform, 
 and $\sum \dim T_k=\dim \check{X}$. So we  may write the second term as  
$q^{(g-1) (\varepsilon-\dim \check{X})/2} = \Delta^{\frac{\varepsilon - \mathrm{dim}(\check{X})}{4}}.$
Also the first term $\det(T_x)(\partial^{-1/2})$ agrees with $\mathfrak{z}$
from \eqref{LXdef}. This shows that \eqref{FromBZSV} coincides with  the definition from \eqref{LXdef} and \eqref{CLR}; they both give
\begin{equation} L_{\check{X}}(\varphi) := \mathfrak{z} \Delta^{\frac{ (\varepsilon - \dim \check{X})}{4}}  \sum_{x \in \mathrm{Fix}(\varphi,\check{X})} \, 
\cdot 
\prod_v \, \mathrm{tr}(\mathrm{Fr}_v \times q^{-1/2} \, | \, \widehat{\mathcal{O}}_{\check{X},x})\end{equation}
where the Euler product is interpreted by analytic continuation.

 \section{Proofs of weak duality in the smooth examples} \label{smoothproofs} 
 
 In this appendix we provide the proof of Theorem \ref{thm: smooth}, although we will
 be somewhat terse; the content here is simply to pin down constants. We adopt the following convention for $L^2$-products: 
 \begin{equation} \label{ipdef} \langle f,g\rangle := \int_{[G]} \, f(x) g(x) \, dx\end{equation} 
 \textit{without} any complex conjugation.
 We will also allow ourselves with integral that diverge because of a central torus simply by formally writing the divergent integral; these computations can all be
 rewritten as absolutely convergent computations by quotienting by the center
 along the lines described in 
\S \ref{sect: cancelZetaOnes}, but the formal structure becomes less elegant when written this way. 
 
 \subsection{First row of the table: Iwasawa--Tate period} 
 \label{Tate}

Here $G_1=\Gm$ acts on $X_1=\mathbf{A}^1$, and $\Ggr$ acts on $X_1$ by scaling, and $G_2, X_2$ identically defined. 
By virtue of this symmetry it's enough to check \eqref{weakduality} in one direction.
We write $(G, X) = (G_1, X_1)$ in what follows.

Note that $\eta(g,\lambda)=g\lambda$, 
corresponding to the factor by which multiplication by $g$ and $\lambda$ scale
$dx \in \Omega^1(\mathbf{A}^1)$.
From \eqref{thetaXdef} and the preceding definitions, we get
 $$
        \theta_X(g)  =  |g \partial^{1/2}|^{1/2} \sum_{x \in F^\times} \Phi^0(x g \partial^{1/2})
 $$

Unfolding in the usual way, we find that for an idele class character  $\chi$ that is nontrivial on ideles of norm $1$ (automatically Whittaker-normalized, for its value at $1$ equals $1$), we have 
\begin{align*} P_{\mathbf{A}^1}(\chi) = \langle \theta_X, \chi \rangle &=   
 \int_{F^{\times} \backslash \mathbb{A}^{\times}} |g \partial^{1/2}|^{1/2} \sum_{x \in F^\times} \Phi^0(x g \partial^{1/2}) \chi(g) dg\\
 &= \chi(\partial^{-1/2})  \int_{F^{\times} \backslash \mathbb{A}^{\times}} |g|^{1/2} \sum_{x \in F^\times} \Phi^0(x g) \chi(g) dg
\\ &= \chi(\partial^{-1/2})L(1/2, \chi)  \stackrel{\eqref{LXdef}}{=} L_{\mathbf{A}^1}( \mbox{parameter of $\chi$}).
  \end{align*}
  The last step is the unramified evaluation in Tate's thesis. 
  This concludes the proof of weak duality in this example.

\subsection{Second row of the table: the group case for $\mathrm{GL}_n$.} \label{gp} Let $H = \mathrm{GL}_n$.
 Let $G_1 = H \times H$, and let $G_2 = \check{H} \times \check{H}$, so that $G_1$ and $G_2$ are Langlands dual. 
  We consider the $G_1$-space $X_1 = H$ with right action given by 
$$x \cdot (h_1,h_2) = h_1^{-1}x h_2$$
and the $G_2$-space $X_2 = \check{H}$ with right action given by
\begin{equation} \label{X2action} x \cdot (h_1, h_2) = {^\iota h}_1^{-1} x h_2 = h_1^T x h_2\end{equation}
where $\iota$ is the involution given by transpose-inverse. We let $\Ggr$ act trivially on both $X_1, X_2$. 

 Again in both cases, taking the Haar measure on $X_1, X_2$, we have
$\eta(h_1,h_2,\lambda) = 1$.  From \eqref{thetaXdef} we get
$$\theta_{X_1}(h_1,h_2) = \Delta^{-\frac{\mathrm{dim}(H)}{4}}\sum_{x \in H(F)} \Phi^0(h_1^{-1}xh_2)$$
Note that $G_1(F)$ acts transitively on $X_1(F)$, and the distinguished base point $\gamma_0 = \mathrm{id}_H \in X_1 $ has stabilizer isomorphic to $H$ 
$$\mathrm{Stab}_{G_1}(\gamma_0) = \big\{(h,h) \in G_1: h \in H\} \cong {^\Delta H}$$
and similarly, $G_2(F)$ acts transitively on $X_2(F)$ with the distinguished base point $\mathrm{id}_H$, whose stabilizer is the $\iota$-twisted diagonal
$$\mathrm{Stab}_{G_2}(\gamma_0) = \big\{(^\iota h, h): h \in H \} \cong {^{\iota\Delta} H}$$
The automorphic $X_1$-period of a cusp form $f_1 \times f_2 \in \pi_1 \boxtimes \pi_2$ on $G_1$ can be computed as 
\begin{align*}
P_{X_1}(f_1 \times f_2) &= \langle \theta_{X_1}, f_1 \times f_2\rangle = \Delta^{-\frac{\mathrm{dim}(H)}{4}}\int_{[G_1]} f_1(h_1)f_2(h_2) \sum_{x \in H(F)} \Phi^0(h_1^{-1}xh_2) \, d(h_1,h_2)\\
&= \Delta^{-\frac{n^2}{4}}  \int_{{^\Delta H}(\mathbb{A}) \backslash G_1(\mathbb{A})} \, \Phi^0(h_1^{-1}h_2) \, \int_{[H]} \, f_1(hh_1)f_2(hh_2) \,  dh \, d(h_1,h_2)\\
&=  \Delta^{-\frac{n^2}{4}} \int_{H(\mathbb{A})} \, \Phi^0(g) \, \int_{[H]} \, f_1(h) f_2(hg) \, dh \, dg  \, \text{ substituting } g \to h_1^{-1}h_2\\
&=  \Delta^{-\frac{n^2}{4}} \, \int_{H(\mathbb{A})} \, \Phi^0(g) \, \big\langle \, f_1, \pi_2(g) f_2 \, \big\rangle_{[H]} \, dg\\
&=   \Delta^{-\frac{n^2}{4}}  \big\langle \, f_1, f_2 \, \big\rangle_{[H]} 
\end{align*} 
where $\langle \, \cdot \, , \, \cdot \, \rangle_{[H]}$ is the inner product on $[H]$ without complex conjugates, as noted in \eqref{ipdef}. 
Up to scaling the automorphic forms, the inner product $\langle \, f_1, f_2\, \rangle_{[H]}$ is nonvanishing precisely when $f_1 = \overline{f_2}$, in other words, when $\pi_1$ is the Hermitian dual of $\pi_2$.

 So suppose $\pi_1 = \pi^{\iota}_2 = \pi$, and $f_1 = f^{\iota}_2 = f \in \pi$ is a Whittaker normalized cusp form on $H$. Then the conjecture of Lapid--Mao which is a theorem for $\mathrm{GL}_n$ (see Theorem 4.1 of \textit{loc. cit}) implies that 
$$\frac{\|f\|^2_{L^2}}{|W_f^0|^2} = \Delta^{n^2/2} q^{\beta_{\mathrm{Whitt}}} \, L(1, \mathrm{Ad}, \pi)$$
as in \cite[\S 14.5]{BZSV}.
(To be precise, $L(s, \mathrm{Ad}, \pi)$ has a pole at $s = 1$ and the $L^2$-norm is infinite; but this equality is understood by formally cancelling the $\zeta(1)$'s on both sides in the sense of Section \ref{sect: cancelZetaOnes}.)

We reprise that argument following 
the notation of Lapid and Mao \cite{Lapid-Mao}. 
In the case at hand their formula states:
$$  \frac{|W_f^0|^2}{ \int_{[G]} |f|^2 }  =  \mathrm{vol}([G])^{-1}
\int_{n\in N(\mathbb{A})} \langle n a_0 f, a_0 f \rangle \psi(n) dn$$
 where  the integral at the end is to be expressed as a product
and regularized in a standard way.
(Specifically, see equation (3.3) and Conjecture 3.3 of \cite{Lapid-Mao}, as well as Proposition 2.3
for some of the notation;  we formally take $S$ to be all places,  $\varphi=a_0\cdot f$, and $\varphi^{\vee}$ its conjugate.)
In our case $\mathrm{vol}([G]) =  \Delta^{n^2/2} \prod_{1}^n \zeta(i)$. 
The measure on $N$ assigns mass $1$ to the quotient, so that $dn = q^{-(g-1) u} \prod_{v} dn_v$
where each $dn_v$ assigns mass $1$ to the integral points of $N$ and $u = \mathrm{dim}(N)$.
Note that in fact $\psi(n) = \psi_0(a_0^{-1} n a_0)$
where $\psi_0$ is everywhere ``unramified,'' and $a_0$
is as in \eqref{a0def}. 
Making the substitution $n \leftarrow a_0 n a_0^{-1}$, 
 so that $d(a_0n a_0^{-1}) =|e^{2\rho(a_0)}| dn$, we compute the integral above to equal
 $$q^{-(g-1) u}  |e^{2\rho(a_0)}| \int \langle \pi(n) v,   v \rangle \psi_0(n) = q^{-(g-1) u + (2g-2) \langle 2 \rho, \check{\rho} \rangle} \frac{\prod_{1}^n \zeta(i)}{L(1, \mathrm{ad})},$$
 where we used the unramified evaluation of the local integrals as in Proposition 2.14 of \textit{loc. cit}. We conclude with \eqref{Whitnorm} to rewrite the power of $q$ in front as $\beta_{\mathrm{Whitt}}$.

 Since $f$ is Whittaker normalized, $|W_f^0|^2$ is exactly $q^{-\beta_{\mathrm{Whitt}}}$, so for such an $f$ we have 
\begin{equation} \label{X1group1} \langle \theta_{X_1},  f \times \bar{f} \rangle = \Delta^{-\frac{n^2}{4}} \, \|f\|_{L^2}^2 = \Delta^{+ \frac{n^2}{4} }\, L(1, \mathrm{Ad}, \pi)\end{equation}

Let's compare this to the spectral $X_2$-period. Note that the $G_2$-valued $L$-parameter $\varphi_1 \times \varphi_2$ of $\pi_1 \boxtimes \pi_2 $
has a fixed point on $X_2$ if and only if $^\iota \varphi_1$ is conjugate to $\varphi := \varphi_2$. In this case,  for cuspidal $\pi_i$, the identity coset in $X_2$ is the unique fix point, whose tangent space is $\check{\mathfrak{h}}$ equipped with the Galois representation $\mathrm{Ad} \circ \varphi$. Since $\Ggr$ acts trivially, we evaluate the adjoint $L$-function at $s = 0$: 
\begin{equation} \label{X1group2} L_{X_2}(f \times f) = \Delta^{-\frac{n^2}{4}} \, L(0, \mathrm{Ad}, \pi) = \Delta^{\frac{n^2}{4}}L(1, \mathrm{Ad}, \pi)\end{equation}
by the functional equation of $L(s, \mathrm{Ad},\pi)$. The agreement between \eqref{X1group1} and \eqref{X1group2} establishes \eqref{weakduality}, i.e. one-half of Theorem \ref{thm: smooth} in this case. 

The automorphic $X_2$-period of a cusp form $f = f_1 \times f_2 \in \pi_1 \boxtimes \pi_2$ on $G_2$ can be unfolded in a similar way
\begin{align*}
P_{X_2}(f) &= \langle \, \theta_{X_2}, f\, \rangle = \Delta^{-\frac{n^2}{4}} \, \int_{{^{\iota\Delta}H}(\mathbb{A}) \backslash G_2(\mathbb{A}) } \Phi^0(^\iota h_1^{-1} h_2) \int_{[H]} \, f_1(^\iota h h_1)f_2(hh_2) \, dh \, d(h_1,h_2)\\
&= \Delta^{-\frac{n^2}{4}} \int_{H(\mathbb{A})} \, \Phi^0(g) \int_{[H]} \, f_1(^\iota h)f_2(hg) \, dh \, dg \, \text{ substituting } g \to {^\iota h}_1^{-1} h_2\\
&=  \Delta^{-\frac{n^2}{4}} \, \int_{H(\mathbb{A})} \, \Phi^0(g) \langle \, f_1^{\iota}, \pi_2(g)f_2 \rangle_{[H]} \, dg\\
&=  \Delta^{-\frac{n^2}{4}} \, \langle f_1^{\iota}, f_2 \rangle
\end{align*}
Up to scaling the automorphic forms, the inner product $\langle f_1^{\iota}, f_2 \rangle$ is nonvanishing precisely when $f_1 = f_2$. So suppose $\pi_1 = \pi_2 = \pi$, and $f_1 = f_2 = f \in \pi$ is a Whittaker normalized cusp form on $H$. Then
$$\langle \, \theta_{X_2}, f \times f\rangle = \Delta^{-\frac{n^2}{4}} \|f\|^2_{L^2} = \Delta^{\frac{n^2}{4}}L(1, \mathrm{Ad}, \pi) = L_{X_1}(\varphi \times \varphi)$$  by the same considerations as the previous computation.  This concludes the proof of Theorem \ref{thm: smooth} in this case. 

 \subsection{Third row of the table: Godement--Jacquet and Rankin--Selberg}
 
We now  consider  ``the Godement--Jacquet period''
where $X_1$ consists of $n \times n$ matrices, acted on $G_1=\GL_n \times \GL_n$ by means of
\begin{equation} \label{matrixactiondef} M(g_1, g_2) = g_1^{-1} M g_2\end{equation} 
and the $\Ggr$ action by scaling.  
On the other side, we will take the ``Rankin--Selberg'' period of Jacquet--Piatetski-Shapiro and Shalika,
wherein $X_2 = \GL_n \times \mathbf{A}^n$
and the action of $G_2 = \GL_n \times \GL_n$ will be via 
  \begin{equation}  \label{GJRSX2def}  (x, v) (g_1, g_2)  = (g_1^T x g_2,  v g_2)\end{equation} 
  and $\Ggr$ acting by scaling on the $\mathbf{A}^n$ factor. Observe here the transpose on $g_1$, cf. \eqref{X2action}.

  \subsubsection{Godement--Jacquet $G_1=\GL_n^2,  X_1= \mathrm{Mat}_{n \times n}$}  \label{GJcase}

We write $(G, X)$ for $(G_1, X_1)$ in \S \ref{GJcase}. 
In this case, 
an invariant top form is scaled by
$$\eta(g_1,g_2,\lambda) = |g_1^{-1}g_2|^n\, |\lambda|^{n^2}$$
where to shorten notation we write $|g|$ in place of $|\det g|_{\mathbb{A}}$.

From \eqref{thetaXdef} we get
 $$
        \theta_{X_1}(g_1, g_2)  =  (\cdots) |\partial|^{n^2/4} |g_1^{-1} g_2|^{n/2} \sum_{x \in \mathring{X}_1(F)} \Phi^0((g_1^{-1} x g_2) \partial^{1/2}).
 $$
 Here $(\cdots)$ is the dimensional factor $\Delta^{-n^2/4}$, and $\mathring{X}_1 \subset X_1$ is the open orbit of matrices of full rank. We continue to abridge it by $(\cdots)$ since later on we are going to compare it to the group case $G=\GL_n^2, X=\GL_n$ that will have precisely the same prefactor.
 
 Let $f_1,f_2$ be two cusp forms. Notice that $\mathring{X}_1(F)$ is the orbit of the identity, whose stabilizer
 is the diagonal copy of $\GL_n(F)$\footnote{Note that in the situation when $f_1, f_2$ are cuspidal, we may replace $\mathring{X}_1(F)$ by $X_1(F)$ in the summation since the only contribution to the integral $\langle \theta_{X_1}, f_1 \times f_2 \rangle$, 
 is from the open orbit $\mathring{X}_1(F) \subset X_1(F)$  (see Lemma 12.13 of \cite{Godement-Jacquet}).}.  Accordingly, we unfold: 
\begin{equation} \label{GJeq1}
 \langle \theta_{X_1}, f_1 \times f_2 \rangle = (\cdots) |\partial|^{n^2/4}  \int_{\GL_n(F)^{\Delta} \backslash \GL_n^2(\mathbb{A})} f_1(g_1) f_2(g_2) \Phi^0 (g_1^{-1} g_2 \partial^{1/2})|g_1^{-1}g_2|^{n/2}. 
 \end{equation}
 which equals, under the change of variables $g_2 \leftarrow g_1g$ and integrating out the $g_1$-variable, 
 $$ =  (\cdots) |\partial|^{n^2/4}   \int_{\GL_n(\mathbb{A})} \langle f_1, g f_2 \rangle_{[\GL_n]} \Phi^0(g \partial^{1/2})|g|^{n/2}.$$
 where $\langle f_1, g f_2 \rangle$ denotes the integral over $[\GL_n]$
 of $f_1 \cdot (g f_2)$.  This integral is divergent, and contains a factor
 of $\zeta(1)$ which we will regularize as previously discussed. 
 
 This integral vanishes unless (up to scaling) $f_1 =\bar{f}, f_2 =f$. In that case, denote
 it by $\mu(g)$, i.e. the  matrix coefficient associated to $f$;
  we get $$  = (\cdots) |\partial|^{n^2/4} \int_{\GL_n(\mathbb{A})} \mu(g) |g|^{n/2} \Phi^0(g \partial^{1/2}),$$
Making the substitution $g \leftarrow g \partial^{1/2}$ eliminates the factor $|\partial|^{n^2/4}$. Moreover,
after thus substituting, 
the integral recovers the standard $L$-function of $f_2$ at $1/2$ (in the sense of analytic continuation, use Lemma 6.10 of \cite{Godement-Jacquet} with $s=1/2$), multiplied by the value of $\mu$ at $\partial^{-1/2}$,
which is precisely $\langle f_1, f_2 \rangle$ multiplied by the central character
$\chi_2(\partial^{-1/2})$; so, all in all, the pairing of
\eqref{GJeq1} equals 
$$ (\cdots)  \langle f_1, f_2 \rangle \chi(\partial^{-1/2}) L(\frac{1}{2}, f_2).$$
Comparing to the normalized group period computed in \eqref{gp} --  that is to say the period
for $X_{\textrm{group}}=\GL_n$ as a $\GL_n \times \GL_n$-space -- which is just
$(\cdots) \langle \varphi, \varphi \rangle,$ we see that
$$ \langle \theta_{X}, f_1 \times f_2  \rangle = \langle \theta_{X^{\textrm{group}}}, f_1 \times f_2 \rangle  \cdot \left( \chi(\partial^{-1/2}) L(\frac{1}{2}, f_2)  \right).$$
By the computation  \eqref{X1group1} in the group case, the first term equals
the normalized $L$-function for the space $Y_{a} = \GL_n$
as $\GL_n \times \GL_n$-space, and the second term
contributes the normalized $L$-function for the space  $Y_b=\mathbf{A}^n$  considered as a $\GL_n \times \GL_n$-space through the standard representation of the second factor;
since $X_2 = Y_a \times Y_b$, see \eqref{GJRSX2def}, 
 the right hand side gives indeed $L_{X_2}(\varphi)$.

 \subsubsection{Rankin--Selberg $G_2 = \GL_n^2, X_2 = \GL_n \times \mathbf{A}^n$. } \label{JPSS}
 As before we write $(G, X)$ for $(G_2, X_2)$ here, with action
 as in \eqref{GJRSX2def}. 

Consider the base point
$$w = \begin{bmatrix}  & & & 1\\ & & -1 & \\  & \reflectbox{$\ddots$} & &\\ (-1)^{n-1} & & & \end{bmatrix}$$
Observe that $w^{-1} = w$ if $n$ is odd, and $w^{-1} = -w$ if $n$ is even. 
In particular, we can compute the stabilizer of $(w \times 0) \in \GL_n \times \mathbf{A}^n$
$$\mathrm{GL}_n^{\tau\Delta} :=
 \mathrm{Stab}_{\mathrm{GL}_n^2}(w,0) = (g, g^\tau ) := w\,  {(g^T)^{-1}} w^{-1}) $$
by the calculation
$$g^T \cdot w \cdot g^\tau = g^T \cdot w \cdot w \, {(g^T)^{-1}} w^{-1} = (-1)^{n-1}w^{-1} = w$$
The eigenform has eigenvalue
  $$\eta(g_1,g_2,\lambda) = |g_2|\, |\lambda|^n$$
  Let $\Phi^0 = \Phi_1 \otimes \Phi_2$ be the standard Schwartz function on $X_2 = \GL_n \times \mathbf{A}^n$, then $\mathring{X}_2(F) = \GL_n(F) \times (F^n-0)$ and from \eqref{thetaXdef} we get
  $$  \theta_{X_2}(g_1, g_2) =  (\cdots) |\partial|^{n/4} |g_2|^{1/2} \sum_{x \in \GL_n(F), y \in F^n-0} \Phi_1(g_1^{T} x g_2) \Phi_2(yg_2 \partial^{1/2})$$
  The dimension factor $(\cdots)$ equals  $\Delta^{\frac{n-n^2}{4}} = 
\Delta^{-\dim(U)/2} = q^{-(g-1) \dim U}$
  where $U$ is the unipotent for $\GL_n$. 
 
 When we do the integral $\langle \theta_{X_2}, f_1 \times f_2 \rangle$ we proceed just as in \S \ref{GJcase} to get
    $$\Delta^{-\dim(U)/2}  |\partial|^{n/4} \int_{\mathrm{GL}_n^{\tau\Delta} (F) \backslash \GL_n(\mathbb{A})^2} 
    \Phi_1(g_1^T w g_2) f_1(g_1) f_2(g_2)  \sum_{y \in F^n-0} \Phi_2(yg_2 \partial^{1/2}) |g_2|^{1/2}.
    $$
Nonvanishing of $\Phi_1$ restricts to the locus where $g_1 = g_2^\tau G_2(\mathcal{O})$ so one gets,  
relabelling $g_2 \leftarrow g$, 
         $$\Delta^{-\dim(U)/2}   |\partial|^{n/4} \int_{\GL_n(F) \backslash \GL_n(\mathbb{A})}
 f_1(g^\tau ) f_2(g)  \sum_{y \in F^n-0} \Phi_2(yg \partial^{1/2}) |g|^{1/2}.
    $$ 
  Folding further via the open $\GL_n(F)$ orbit
  of $$y = (0,0,\dots, 1),$$ with stabilizer the mirabolic subgroup $P_n$
  with bottom row $(0,\dots, 1)$ we get:
         $$\Delta^{-\dim(U)/2} |\partial|^{n/4} \int_{P_n(F) \backslash \GL_n(\mathbb{A})}
 f_1^{\tau}(g) f_2(g)  \Phi_2(yg \partial^{1/2}) |g|^{1/2} dg.
    $$ 
  where we wrote $f_1^{\tau}(g) = f_1(g^\tau)$. Now unfold via $f_i(g) = \sum_{\gamma \in N\backslash P_n(F)} W_i(\gamma g)$,
with the Whittaker functions $W_i$ of $f_i$ being normalized using probability measures
on the unipotent subgroups. Note that  $\tau$ stabilizes $N$ and also the Whittaker character. 
Also we should use in fact inverse characters for the Whittaker functions $W_1, W_2$; this will make no difference
to our computation, so we will not explicitly keep track of it in the notation. 

   Then the above becomes
    $$ \Delta^{-\dim(U)/2} |\partial|^{n/4} \int_{N_F \backslash \GL_n(\mathbb{A})} W_1^{\tau}(g) W_2(g) \Phi(y g \partial^{1/2}) |g|^{1/2} dg.$$
    Since the entire integrand is $N(\mathbb{A})$-invariant, we may replace the domain of integration by $N(\mathbb{A}) \backslash \GL_n(\mathbb{A})$ at the cost of multiplying by the volume of $[N]$ with the integrally normalized measure, which exactly cancels the dimension factor $\Delta^{-\dim(U)/2}$ in front. In other words, the above can be written as an Euler-factorizable expression
    $$ |\partial|^{n/4} \int_{N(\mathbb{A}) \backslash \GL_n(\mathbb{A})} W_1^{\tau}(g) W_2(g) \Phi(y g \partial^{1/2}) |g|^{1/2} dg$$
    We proceed with the calculation of the local factors
    $$ |\partial_v|^{n/4} \int_{N_v \backslash \GL_n(F_v)} W_{1,v}^{\tau}(g) W_{2,v}(g) \Phi(y g \partial_v^{1/2}) |g|_v^{1/2} dg_v$$
    although we will often drop the subscripts $v$ to simplify notation.
    
    First of all, use Iwasawa decomposition to write the integral over $A_v$ (remembering that the Iwasawa decomposition of the measure $dg$ is given by $dg = |e^{-2\rho}|dn\, da\, dk$), and we use  \eqref{shifted CS} to write
    $W_i (a_v)= W_i^0 \cdot W_{i}^{\mathrm{un}}(a_{0,v}^{-1} a_v)$. Then we have
    $$|\partial|^{n/4} W_1^0 W_2^0 \prod_v  \int_{A_v} W_1^{\mathrm{un},\tau}(a_0^{-1}a) W_2^{\mathrm{un}}(a_0^{-1}a)|e^{-2 \rho}(a)| \Phi(y  \partial^{1/2}    a) |a|^{1/2} da$$
    Next, we translate $a$ by $a_0$ in the above integral. We calculate from the definition \eqref{a0ex} of $a_0$ that
    $$|e^{-2\rho}(a_0)| = |\partial^{1/2}|^{\langle 2\rho, 2\check{\rho}\rangle} = \Delta^{-2\langle \rho, \check{\rho}\rangle}$$
    and notice that $a_0$ acts on $y$ by its last entry, which is $\partial^{(n-1)/2}$. Also, since $\mathrm{det}(a_0)=1$ we have also $|a_0| = 1$. Thus, this translation turns our integral into
    $$|\partial|^{n/4} W_1^0 W_2^0 \Delta^{-2\langle \rho, \check{\rho}\rangle}  \prod_v \int_{A_v} W_1^{\mathrm{un},\tau}(a) W_2^{\mathrm{un}}(a)|e^{-2 \rho}(a)| \Phi(y  \partial^{n/2}    a) |a|^{1/2} da$$
    Now we want to translate $a$ by the central element $\partial^{-n/2}$. Writing $\chi_i$ for the central characters of $W_i$, we see that 
    $$W_1^{\mathrm{un}, \tau}(\partial^{-n/2}a)W_2^{\mathrm{ur}}(\partial^{-n/2} a) = \chi_1^{-1}\chi_2(\partial^{-n/2})W_1^{\mathrm{un},\tau}(a) W_2^{\mathrm{un}}(a)$$
     and the central character factors will come out in front. The term $|a|^{1/2}$ becomes $|\partial^{-n/2} a|^{1/2} = |\partial|^{-n^2/4}|a|^{1/2}$; combining this with the $|\partial|^{n/4}$ in front we get a factor of  $|\partial|^{(n-n^2)/4} = \Delta^{\mathrm{dim}(U)/2}$. Finally, also note that this central translation does not affect the term $|e^{-2\rho}(a)|$, so in the end we have
    $$\chi_1^{-1}\chi_2(\partial^{-n/2})W_1^0 W_2^0 \Delta^{\mathrm{dim}(U)/2-2\langle \rho, \check{\rho}\rangle}  \prod_v \int_{A_v} W_1^{\mathrm{un},\tau}(a) W_2^{\mathrm{un}}(a)|e^{-2 \rho}(a)| \Phi(ya) |a|^{1/2} da$$
   Since we are using Whittaker normalization  -- see \eqref{Whitnormnew} -- we have that the prefactor
   $$ \Delta^{\dim(U)/2- 2\langle  \rho,  \check{\rho} \rangle)}
   W_1^0 W_2^0 = 1,$$
   so we arrive at 
   \begin{align*}&\chi_1^{-1}\chi_2(\partial^{-n/2})\prod_v \int_{A_v} W_1^{\mathrm{un},\tau}(a) W_2^{\mathrm{un}}(a)|e^{-2 \rho}(a)| \Phi(ya) |a|^{1/2} da\\
  &= \chi_1^{-1}\chi_2(\partial^{-n/2}) L_v(\frac{1}{2}, f_1^{\tau} \times f_2).
  \end{align*}
For the unramified local integral calculation, see, for example, Proposition 2.3 of \cite{Jacquet-Shalika}. 

After taking product over $v$, the product of these local $L_v$  give $L_{X_1}(f_1 \times f_2)$, taking into account the action of $\GL_n \times \GL_n$ on matrices given by 
\eqref{matrixactiondef}: note that $\tau$ acts on diagonal matrices by $(a_1,\ldots, a_n) \mapsto (a_n^{-1}, \ldots, a_1^{-1})$, so it takes a Satake parameter indexing a representation of the dual group to the dual representation.

 \subsection{Final row of the table: Hecke and Eisenstein periods on $\GL_2$}
 
 Finally, we consider $G_1=\GL_2$ acting on $X_1=T \backslash \GL_2$ on one side,
 with $T$ the torus of diagonal matrices $\mathrm{diag}(x, 1)$, and 
 $X_1$ has  trivial $\Ggr$ action; and on the other side $G=\GL_2$ acting on $X_2=\mathbf{A}^2$
 with scaling $\Ggr$ action.
 
\subsubsection{Hecke period} \label{Hecke}
We take $(G, X) = (G_1, X_1)$. 
 Here 
$$P_X^{\norm}(f) = \Delta^{-1/4} \int_{[T]} f(t) dt = \Delta^{-1/4} \int_{T(\mathbb{A})} W_f(t) dt
=   \prod_{v} P_{X, v},$$
where we unfolded, using the fact (Whittaker normalization, \eqref{Whitnorm}) that $W_f^0 = \Delta^{1/4}$, to a product
of local factors given by: 
\begin{align*}
P_{X,v} &= \int_{T(F_v)} W_v^{\mathrm{un}}(t_v a_{0,v}^{-1})  dt\\
&=  \sum_{\alpha} \,  q_v^{-(\alpha+2m_v)/2} s_{\alpha+m_v,-m_v}(\mathrm{Sat}_{v})\\
&= L_v(1/2, f) \times \chi_v(\varpi_v^{-m_v})
\end{align*}
where $m_v$ is the order of the choice of half different at $v$, as in \eqref{partialdef}. Here, $\alpha$ ranges through integers $\geq 0$ and $s_{\alpha,\beta}(\mathrm{Sat}_v)$
refers to the trace of Satake parameter at $v$ in the representation of the dual group $\check{G} = \GL_2$
with highest weight $(\alpha, \beta)$; at the last step, we reindexed via $\alpha \leftarrow \alpha+2m_v$.  
Taking the product we get $\chi_f(\partial^{-1/2}) L(\frac{1}{2}, f) $ which 
is just $L_{\mathbf{A}^2}(\varphi)$ as defined in \eqref{LXdef}.

\subsubsection{$G_2=\SL_2, X_2=\mathbf{A}^2$} \label{Eiscase}
Strictly speaking, according to our definition of ``weak duality,''
there is nothing to compute: the pairing of $\theta_{X_2}$
with any cusp form vanishes. However, let us compute 
it for an Eisenstein series to show that the underlying phenomenon remains valid there too
after using an obvious regularization. 

This amounts to a computation 
of constant terms of Eisenstein series. 
 Write $\psi_{s}$ for the spherical function on $G(\mathbb{A})$ whose value at $NAK$
is given by $|e^{\rho}(a)|^{1+2s}$.
Let $E_s$ be the Whittaker-normalied Eisenstein series induced from $\psi_s$; this means we form
 first $\sum_{B_F \backslash G_F} \psi_s(\gamma g)$
(absolutely convergent for $\Re(s)$ large enough), and then scale it so that it is Whittaker normalized in the sense
of \eqref{Whitnorm}). 

We have
$$ \theta_X(g) = \Delta^{-1/4} |\partial^{1/2}| \sum_{v \in F^2-0}  \Phi( v  \partial^{1/2}   g).$$
 The prefactor equals
 $\Delta^{-3/4}$. 
 We proceed to compute
$ \langle \theta_X,  f \rangle$
for an automorphic form $f$ orthogonal to the constants. 
 
We get:
\begin{eqnarray*}
 \langle \theta_X,  f\rangle = \Delta^{-3/4} \int_{[\SL_2]} \sum_{\gamma \in N_F \backslash \SL_2(F)} \Phi( (0,1) \partial^{1/2} \gamma g) f(g) \\ =  \Delta^{-1/4}
\int_{N_{\mathbb{A}} \backslash \SL_2(\mathbb{A})} \Phi((0,1)  \partial^{1/2} g) f^N(g)
\end{eqnarray*}
using the fact that the measure of  $N_F \backslash N_{\mathbb{A}}$ equals $\Delta^{1/2}$; and $f^N$ is the 
  constant term along $N$, normalized using the probability Haar measure. 
  
   In particular, the period vanishes if $f$ is cuspidal; this is just as requied by weak duality,
   because the Galois parameter of $f$ fixes no points on $X_1$ in that case. 
Now, proceeding formally, let us see that a suitably regularized version of the desired equality
continues to hold when $f$ is Eisenstein.  {\em We will proceed formally, ignoring all issues of convergence,}
although it is not difficult to set up a regularization scheme to make rigorous sense of what follows. 

Let $E_s$ be the Whittaker-normalized Eisenstein series induced from $\psi_s$.
Below we will compute that  
 \begin{equation} \label{ECnormalized} E_s^N =  \left( \xi^{\norm}(-2s) \psi_s + \xi^{\norm}(2s) \psi_{-s}\right)\end{equation}
where $\xi^{\norm}$ is the normalized
$\zeta$-function for the curve $\Sigma$, explicitly given by
$\xi^{\norm}(s) = \Delta^{ \frac{s-1/2}{2}}  \zeta(s)$, which is then symmetric under $s \leftrightarrow 1-s$. 

Assuming Equation \eqref{ECnormalized}  for now let us proceed with the period computation. 
Consider for example the $\psi_s$ term from \eqref{ECnormalized}; it contributes to $\langle \theta_X,  f \rangle$ the 
quantity $\Delta^{-1/4} I(s) \xi^{\norm}(-2s)$, where  $I(s)$
equals $\int \Phi((0,1) \partial^{1/2} g) \psi_s(g)$, which, upon
expanding in Iwasawa coordinates, becomes (multiplicative measure on $y$ in the first two integrals):
\begin{multline}  I(s) = \int_{\partial^{1/2} y^{-1} \in \mathcal{O}} |y|^{1+2s} |y|^{-2}  \stackrel{y \leftarrow 1/y}{=}
\int_{y \in \mathcal{O} \partial^{-1/2}} |y|^{1-2s}   = |\partial|^{(s-1/2)} \zeta(1-2s)  \\ = \Delta^{1/2-s} \zeta(1-2s) =  \Delta^{1/4} \xi^{\norm}(1-2s)
 =  \Delta^{1/4} \xi^{\norm}(2s). \end{multline}
Therefore, the contribution of $\psi_s$ to $\langle \theta_X, f \rangle$ is $\xi^{\norm}(2s) \xi^{\norm}(-2s)$;
the conrtribution of the other term is exactly the same, and so 
  $$P_X(E_s) = 2 \xi^{\norm}(2s) \xi^{\norm}(-2s).$$ 
  
    Note that $\xi^{\norm}(2s) \xi^{\norm}(-2s) =  \Delta^{-1/2} \zeta(2s) \zeta(-2s)$.
    Therefore, this agrees with the normalized spectral period. Indeed, that spectral period, 
as defined in \eqref{LXdef}
  has a sum over two fixed points, and each of them contribute $\Delta^{-1/2} \zeta(2s) \zeta(-2s)$.

  \proof (of \eqref{ECnormalized}).  
Write $W^{\norm} = \Delta^{-1/4} W^0_{E_s}$.   We must compute $\frac{E_s}{W^{\norm}}$ and its constant term.
From   
\eqref{Whitdef} we get  after unfolding 
 $$W^{\norm} = \Delta^{-1/4} \int_{N(\adele)} \psi(u) \psi_{s}(w u a_0) du$$
 with $a_0$ as in \eqref{a0ex}.  The
 integral over $N(\adele)$ is with respect to the probability Haar measure,
 which equals  $\Delta^{-1/2}$ times the  integrally normalized measure. 
  Now $\psi_{s}(w u a_0)
 = \psi_{s}(w a_0 u^{a_0}) = |e^{\rho}(a_0)|^{-1-2s} 
\psi_{s}(w u^{a_0})$ where $( \,\cdot \,)^{a_0} := a_0^{-1} ( \, \cdot \, ) a_0$. Also $|e^{\rho}(a_0)| = |\partial^{-1/2}| =  \Delta^{1/2}$, see \eqref{a0ex}. 
Substituting $u \leftarrow \mathrm{Ad}(a_0) u$
incurs an extra measure normalization factor of $|e^{2\rho}(a_0)| =\Delta$. 
In total we get
  $$W^{\norm} = \Delta^{-3/4} \cdot \Delta \cdot \Delta^{(-1-2s)/2} [\dots]
 = \Delta^{-1/4-s} [\dots]$$

 where $[\dots]$ refers to the integral $\int_{N(\mathbb{A})}  \psi(a_0 u a_0^{-1}) \psi_s(w u) du$, the integral
 being taken with respect to the integrally normalized measure.
 Noting that $\psi(a_0 u a_0^{-1})$ is an unramified character,  
 the integral evaluates to $\zeta(2s+1)^{-1}$;
 thus $W^{\norm}(E_{s})$ is given by $\Delta^{-\frac{1}{2} (1/2+2s)} \xi(2s+1)^{-1} = \xi^{\norm}(2s+1)^{-1}$. 
 Now, using the evaluation of the (probability-measure-normalized)
 constant term as $\psi_s   
  + \frac{\xi^{\norm}(2s)}{\xi^{\norm}(2s+1)} \psi_{-s}$, we get
 \begin{align*}
\left[\frac{ E_{s}}{W^{\norm}} \right]_U &= 
 \left(\psi_s   
  + \frac{\xi^{\norm}(2s)}{\xi^{\norm}(2s+1)} \psi_{-s} 
\right)   \xi^{\norm}(2s+1) \\ 
 &= \xi^{\norm}(2s+1) \psi_s + 
 \xi^{\norm}(2s) \psi_{-s} 
 \end{align*}
 which coincides with the desired expression \eqref{ECnormalized} after using the functional equation.

\end{document}